\pdfoutput=1 
\DeclareSymbolFont{AMSb}{U}{msb}{m}{n}
\documentclass[11pt,noamsfonts,a4paper]{amsart}
\usepackage[left=1in, right=1in, top=1in, bottom=1in]{geometry}
\usepackage{mathtools}
\usepackage{braket}
\usepackage[charter]{mathdesign}
\usepackage[tracking]{microtype}
\usepackage{tabularx}
\usepackage{etoolbox}
\usepackage{graphicx}
\usepackage{stmaryrd} 

\usepackage{amsmath,amsthm}
\newtheoremstyle{pineapple}%
  {1em}{1em}%
  {\itshape}{}%
  {\bfseries}{. ---}
  {0.5em}{}

\newtheoremstyle{durian}%
  {1em}{1em}%
  {}{}%
  {\bfseries}{. ---}
  {0.5em}{}

\makeatletter
\def\swappedhead#1#2#3{%
  \thmnumber{\@upn{\the\thm@headfont#2\@ifnotempty{#1}{.~}}}%
  \thmname{#1}%
  \thmnote{ {\the\thm@notefont(#3)}}}
\makeatother

\makeatletter
\def\@sect#1#2#3#4#5#6[#7]#8{%
  \edef\@toclevel{\ifnum#2=\@m 0\else\number#2\fi}%
  \ifnum #2>\c@secnumdepth \let\@secnumber\@empty
  \else \@xp\let\@xp\@secnumber\csname the#1\endcsname\fi
  \@tempskipa #5\relax
  \ifnum #2>\c@secnumdepth
    \let\@svsec\@empty
  \else
    \refstepcounter{#1}%
    \edef\@secnumpunct{%
      \ifdim\@tempskipa>\z@ 
        \@ifnotempty{#8}{.~}%
      \else
        \@ifempty{#8}{.}{.~}%
      \fi
    }%
    \@ifempty{#8}{%
      \ifnum #2=\tw@ \def\@secnumfont{\bfseries}\fi}{}%
    \protected@edef\@svsec{%
      \ifnum#2<\@m
        \@ifundefined{#1name}{}{%
          \ignorespaces\csname #1name\endcsname\space
        }%
      \fi
      \@seccntformat{#1}%
    }%
  \fi
  \ifdim \@tempskipa>\z@ 
    \begingroup #6\relax
    \@hangfrom{\hskip #3\relax\@svsec}{\interlinepenalty\@M #8\par}%
    \endgroup
    \ifnum#2>\@m \else \@tocwrite{#1}{#8}\fi
  \else
  \def\@svsechd{#6\hskip #3\@svsec
    \@ifnotempty{#8}{\ignorespaces#8\unskip
       \@addpunct.}%
    \ifnum#2>\@m \else \@tocwrite{#1}{#8}\fi
  }%
  \fi
  \global\@nobreaktrue
  \@xsect{#5}}
\makeatother

\makeatletter
\def\@seccntformat#1{%
  \protect\textup{\protect\@secnumfont
    \ifnum\pdfstrcmp{subsection}{#1}=0 \bfseries\fi
    \csname the#1\endcsname
    \protect\@secnumpunct
  }%
}
\makeatother

\theoremstyle{pineapple}
\newtheorem{IntroTheorem}{Theorem}

\swapnumbers
\newtheorem{Theorem}[subsection]{Theorem}
\newtheorem{Lemma}[subsection]{Lemma}
\newtheorem{Proposition}[subsection]{Proposition}
\newtheorem{Corollary}[subsection]{Corollary}

\theoremstyle{durian}

\usepackage[dvipsnames]{xcolor}
\usepackage{hyperref} 
\hypersetup{
  colorlinks = true,
  linkcolor = Fuchsia,
  urlcolor = Fuchsia,
  citecolor = ForestGreen,
  linkbordercolor = {white}}
\usepackage{tikz-cd}
\usetikzlibrary{arrows}

\tikzset{
  symbol/.style={
    draw=none,
    every to/.append style={
      edge node={node [sloped, allow upside down, auto=false]{$#1$}}}
  }
}

\newcommand\scalemath[2]{\scalebox{#1}{\mbox{\ensuremath{\displaystyle #2}}}}

\usepackage{enumitem}
\setlist[1]{labelindent=\parindent}
\setlist[1]{labelsep=0.5em}
\setlist[enumerate,1]{label={\upshape (\roman*)}, ref={\upshape (\roman*)}}

\makeatletter
\newcommand{\leqnomode}{\tagsleft@true\let\veqno\@@leqno}
\newcommand{\reqnomode}{\tagsleft@false\let\veqno\@@eqno}
\makeatother

\usepackage[utf8]{inputenc}

\tikzset{>={Straight Barb[length=2pt,width=4pt]}, commutative diagrams/arrow style=tikz}

\makeatletter
\let\c@equation\c@subsection

\makeatother

\DeclareMathOperator{\Lie}{Lie}
\DeclareMathOperator{\Tor}{Tor}
\DeclareMathOperator{\Fr}{Fr}
\DeclareMathOperator{\id}{id}
\DeclareMathOperator{\Spec}{Spec}
\DeclareMathOperator{\Sym}{Sym}
\DeclareMathOperator{\Aut}{Aut}
\DeclareMathOperator{\GL}{\mathbf{GL}}
\DeclareMathOperator{\AutSch}{\mathbf{Aut}}

\DeclareMathOperator{\image}{im}

\makeatletter
\newcommand*{\coloneqq}{\mathrel{\rlap{%
           \raisebox{0.3ex}{$\m@th\cdot$}}%
           \raisebox{-0.3ex}{$\m@th\cdot$}}%
           =}
\newcommand{\eqqcolon}{=%
           \mathrel{\rlap{%
           \raisebox{0.3ex}{$\m@th\cdot$}}%
           \raisebox{-0.3ex}{$\m@th\cdot$}}}
\makeatother

\newcommand{\punct}[1]{\makebox[0pt][l]{\,#1}} 
\newcommand{\parref}[1]{{\bf\ref{#1}}}

\DeclareMathOperator{\Gram}{Gram}
\DeclareMathOperator{\rank}{rank}
\DeclareMathOperator{\corank}{corank}
\DeclareMathOperator{\Hom}{Hom}

\DeclareMathOperator{\coker}{coker}

\newcommand{\kk}{\mathbf{k}}
\newcommand{\sO}{\mathcal{O}}
\newcommand{\qbics}{q\operatorname{\bf\!-bics}}
\newcommand{\hrefSP}[1]{\href{https://stacks.math.columbia.edu/tag/#1}{#1}}
\newcommand{\citeSP}[1]{\cite[\hrefSP{#1}]{stacks-project}}

\makeatletter
\newcommand{\smallbullet}{} 
\DeclareRobustCommand\smallbullet{%
  \mathord{\mathpalette\smallbullet@{0.5}}%
}
\newcommand{\smallbullet@}[2]{%
  \vcenter{\hbox{\scalebox{#2}{$\m@th#1\bullet$}}}%
}
\makeatother

\newcommand{\subsectiondash}[1]{\subsection{#1}\textbf{---}\;}

\title{\(q\)-bic forms}
\author{Raymond Cheng}
\address{Institute of Algebraic Geometry \\
  Leibniz University Hannover \\
  Welfengarten 1 \\
  30167 Hannover \\
  Germany
}
\email{cheng@math.uni-hannover.de}

\begin{document}
\begin{abstract}
A \emph{\(q\)-bic form} is a pairing \(V \times V \to \mathbf{k}\) that is
linear in the second variable and \(q\)-power Frobenius linear in the first;
here, \(V\) is a vector space over a field \(\mathbf{k}\) containing the finite field
\(\mathbf{F}_{q^2}\). This article develops a geometric theory of \(q\)-bic
forms in the spirit of that of bilinear forms. I find two filtrations
intrinsically attached to a \(q\)-bic form, with which I define a series of
numerical invariants. These are used to classify, study automorphism group
schemes of, and describe specialization relations in the parameter space of
\(q\)-bic forms.
\end{abstract}
\maketitle
\setcounter{tocdepth}{1}

\thispagestyle{empty}
\section*{Introduction}
Eschewing the semi-linear description of the abstract, a
\emph{\(q\)-bic form} is a bilinear pairing \(\beta\) between a
finite-dimensional vector space \(V\) over a field \(\kk\) and its
\emph{Frobenuis twist} \(V^{[1]} \coloneqq \kk \otimes_{\Fr, \kk} V\), where the
scalar action is twisted by the \(q\)-power Frobenius map
\(\Fr \colon \kk \to \kk\). Equivalently, this is a linear map
\[
\beta \colon V^{[1]} \otimes_\kk V \to \kk.
\]
This article develops an intrinsic theory of \(q\)-bic forms, with an eye
towards the geometry of their hypersurfaces of isotropic vectors, the
\emph{\(q\)-bic hypersurfaces} described in more detail below; briefly, these
are a large class of special hypersurfaces which are important to study as they
exemplify many poorly understood positive characteristic phenomena, such as
supersingularity \cite{Tate:Conjecture, SK:Fermat},
unirationality despite being of general type \cite{Shioda:Fermat, Shioda:Unirationality},
Fano varieties that do not carry minimal free rational curves \cite{Conduche, Shen:Fermat, ratconn}, and
inseparable Gauss maps \cite{Wallace:Duality, Noma, KP:Gauss}.
They are extremal in many respects, such as in relation to \(F\)-singularity
measures \cite{KKPSSW:F-Pure}, and arise in relation to with Deligne--Lusztig
theory \cite{Lusztig:Green, DL, Li:DL}, unitary Shimura varieties
\cite{Vollaard, LZ:Kudla, LTXZZ}, and finite Hermitian geometries
\cite{BC:Hermitian, Segre:Hermitian, Hirschfeld:Geometries}. The theory
developed here opens up new deformation- and moduli-theoretic techniques in the
study of these hypersurfaces, already yielding new results and analogies, as
explored in the companion papers \cite{fano-schemes, qbic-threefolds}.

To begin to appreciate the main aspects of the theory of \(q\)-bic forms,
assume from now on that the base field \(\kk\) contains the finite field
\(\mathbf{F}_{q^2}\), so that the \(q\)-power Frobenius is a nontrivial
endomorphism of \(\kk\). A \(q\)-bic form \(\beta\) then pairs two distinct
\(\kk\)-vector spaces which are nonetheless related by the canonical
\(q\)-linear map \(V \to V^{[1]}\) given by \(v \mapsto v^{[1]} \coloneqq 1
\otimes v\). To emphasize: this map is not \(\kk\)-linear and is not a
bijection if \(\kk\) is imperfect. In any case, a basis \(V = \langle
e_1,\ldots,e_n \rangle\) induces a basis \(V^{[1]} = \langle
e_1^{[1]},\ldots,e_n^{[1]} \rangle\) and the \(q\)-bic form \(\beta\) may be
explicitly understood through its \emph{Gram matrix}:
\[
\Gram(\beta; e_1,\ldots,e_n) \coloneqq
\Big(\beta(e_i^{[1]}, e_j)\Big)_{i,j = 1}^n.
\]
Moreover, if \((e_1',\ldots,e_n') = (e_1,\ldots,e_n) \cdot A^\vee\) is another
basis of \(V\), related to the original via an invertible matrix \(A\), then
the two Gram matrices are related by \emph{\(q\)-twisted conjugation}:
\[
\Gram(\beta; e_1',\ldots,e_n') =
A^{[1],\vee} \cdot \Gram(\beta; e_1,\ldots,e_n) \cdot A
\]
where \(A^{[1],\vee}\) is obtained from the transpose of \(A\) by raising
each matrix entry to the \(q\)-th power. In concrete terms, this article
studies the invariants, orbits, and stabilizers of \(q\)-twisted conjugation.

Returning to the intrinsic situation, contemplate: what are invariants
of \(\beta\)? The key observation is that the Frobenius map makes it
possible to iteratively take left and right orthogonals, thereby canonically
attaching two sequences of vector spaces to \(\beta\). Left orthogonals give
rise to a sequence of subspaces of various Frobenius twists of \(V\), collectively
referred to as the \emph{\(\perp^{[\smallbullet]}\)-filtration};
for instance, the first space is the left kernel
\(\mathrm{P}_1' V^{[1]} = V^\perp \coloneqq \ker(\beta^\vee \colon V^{[1]} \to V^\vee)\).
Right orthogonals give rise to a finite filtration on \(V\),
the \emph{\(\perp\)-filtration}, which has the form
\[
\mathrm{P}_{\smallbullet} V \colon
\{0\} \eqqcolon
\mathrm{P}_{-1} V \subseteq
\mathrm{P}_1 V \subseteq
\mathrm{P}_3 V \subseteq
\cdots \subseteq
\mathrm{P}_- V \subseteq
\mathrm{P}_+ V \subseteq
\cdots \subseteq
\mathrm{P}_4 V \subseteq
\mathrm{P}_2 V \subseteq
\mathrm{P}_0 V \eqqcolon V
\]
where the odd- and even-indexed pieces form increasing and decreasing filtrations,
limiting to subspaces \(\mathrm{P}_- V\) and \(\mathrm{P}_+ V\), respectively;
this time, the first piece is
\(\mathrm{P}_1 V = V^{[1],\perp} \coloneqq \ker(\beta \colon V \to V^{[1],\vee})\).
Numerical invariants of \(\beta\) are then obtained by taking dimensions of
graded pieces of this filtration:
for each \(m \geq 1\), set \(\epsilon \coloneqq (-1)^m\), and
\[
a \coloneqq \dim_\kk \mathrm{P}_+ V/ \mathrm{P}_- V,
\qquad
a_m
\coloneqq \dim_\kk \mathrm{P}_{m-\epsilon-1} V/\mathrm{P}_{m+\epsilon-1} V,
\qquad
b_m \coloneqq a_m - a_{m+1}.
\]

The first result of this article classify \(\beta\) in terms of its
\emph{type} \((a;b_m)_{m \geq 1}\) and properties of its
\(\perp^{[\smallbullet]}\)-filtration. This is simplest over an algebraically
closed field, where \(\beta\) is completely classified by its type. The
finitely many isomorphism classes may be described in terms of a normal form
for its Gram matrix. In the following statement, write \(\mathbf{1}\) for the
\(1\)-by-\(1\) identity matrix, \(\mathbf{N}_m\) for the \(m\)-by-\(m\) Jordan
block with \(0\) on the diagonal, and \(\oplus\) for the block diagonal sum of
matrices:

\begin{IntroTheorem}\label{classification-theorem}
Let \((V,\beta)\) be a \(q\)-bic form of type \((a;b_m)_{m \geq 1}\) over an
algebraically closed field \(\kk\). Then there exists a basis
\(V = \langle e_1,\ldots,e_n \rangle\) such that
\[
\Gram(\beta; e_1,\ldots,e_n) =
\mathbf{1}^{\oplus a} \oplus
\Big(\bigoplus\nolimits_{m \geq 1} \mathbf{N}_m^{\oplus b_m} \Big).
\]
\end{IntroTheorem}

See \parref{forms-classification-theorem-kbar}. Over an arbitrary
field \(\kk\), such a basis exists if and only if the
\(\perp^{[\smallbullet]}\)-filtration admits a purely inseparable descent to a
\(V\), and if a certain Galois cohomology class for a finite unitary group
vanishes: see \parref{forms-classification-theorem} and
\parref{forms-hermitian-diagonal}. The power of the result lies in the fact
that the invariants \((a;b_m)_{m \geq 1}\) are readily computable, yielding an
easy way to determine the isomorphism class of a \(q\)-bic form.

The basis exhibiting \(\beta\) in the normal form above is never unique,
meaning that \(q\)-bic forms admit many automorphisms. The next result
describes the basic structure of the automorphism group
\emph{scheme} \(\AutSch_{(V,\beta)}\) of a \(q\)-bic form \((V,\beta)\). This
group is typically of positive dimension and nonreduced; the following computes
its dimension and that of its Lie algebra in terms of the type of \(\beta\):

\begin{IntroTheorem}\label{automorphisms-theorem}
Let \((V,\beta)\) be a \(q\)-bic form of type \((a;b_m)_{m \geq 1}\) over a
field \(\kk\). Then
\begin{align*}
\dim_\kk\Lie\AutSch_{(V,\beta)} & =
\dim_\kk\Hom_\kk(V,V^{[1],\perp}) =
\Big(a + \sum\nolimits_{m \geq 1} mb_m\Big)\Big(\sum\nolimits_{m \geq 1} b_m\Big),\;\;\text{and} \\
\dim\AutSch_{(V,\beta)} & =
\sum\nolimits_{k \geq 1}\Big[ k(b_{2k-1}^2 + b_{2k}^2) +
\Big(a + \sum\nolimits_{m \geq 2k} m b_m\Big) b_{2k-1} +
2k\Big(\sum\nolimits_{m \geq 2k+1} b_m\Big) b_{2k}\Big].
\end{align*}
\end{IntroTheorem}

The Lie algebra statement is \parref{forms-aut-tangent-space}, where the two
vector spaces are in fact canonically identified, and the second dimension
statement is \parref{aut-dimension}. The proof proceeds by identifying the
reduced subgroup scheme of \(\AutSch_{(V,\beta)}\), say when \(\kk\)
is perfect, with the subgroup of automorphisms that additionally preserve a
descent of the \(\perp^{[\smallbullet]}\)-filtration to \(V\), see
\parref{auts-identify-reduced}; the latter is shown to be reduced in
\parref{auts-smooth-parabolic} by studying infinitesimal deformations of the
identity automorphism. This relies on an observation
\parref{aut-small-lift-along-frob}, which may be of independent interest, that
produces canonical lifts of infinitesimal deformations, up to a Frobenius
twist, in certain situations.

Consider now the parameter space of \(q\)-bic forms on the fixed \(n\)-dimensional
vector space \(V\):
\[
\qbics_V \coloneqq
\mathbf{A}(V^{[1]} \otimes_\kk V)^\vee \coloneqq
\Spec\Sym^*(V^{[1]} \otimes_\kk V).
\]
Via the Gram matrix construction, this may be identified with the affine
space of \(n\)-by-\(n\) matrices. Orbits of the natural
\(\mathbf{GL}_V\) action, corresponding to \(q\)-twisted conjugation of
matrices, are determined by Theorem \parref{classification-theorem} as the
locally closed subschemes
\[
\qbics_{V,\operatorname{type}(\beta)} \coloneqq
\Set{[\beta'] \in \qbics_V | \operatorname{type}(\beta) = \operatorname{type}(\beta')}
\]
parameterizing \(q\)-bic forms with the same type as a given form \(\beta\);
as usual, this is a smooth and irreducible subscheme, and its
codimension is \(\dim\AutSch_{(V,\beta)}\), which is determined by Theorem
\parref{automorphisms-theorem}.

Taken together, these subschemes  yield a finite stratification of the space
\(\qbics_V\). The final results of this article partially characterize the
closure relations amongst the strata in terms of inequalities amongst types.
For a \(q\)-bic form \(\beta\) of type \((a;b_m)_{m \geq 1}\) and each integer
\(m \geq 1\), write
\[
\Psi_m(\beta) \coloneqq
\begin{dcases*}
b_{2k-1}  + 2\sum\nolimits_{\ell = 1}^{k - 1} (k-\ell) b_{2\ell - 1}  &
if \(m = 2k-1\), \\
\sum\nolimits_{\ell = 1}^{k - 1} \ell b_{2\ell} +
k\Big(\sum\nolimits_{\ell \geq 1} b_{2\ell - 1} + \sum\nolimits_{\ell \geq k} b_{2\ell}\Big) &
if \(m = 2k\),
\end{dcases*}
\]
and \(\Theta_m(\beta) \coloneqq \sum\nolimits_{k = 1}^m b_{2k-1}\). With this
notation, the result is as follows:

\begin{figure}
\[
\begin{tikzcd}[row sep=0.1em, column sep=0.8em]
  \mathbf{1}^5 \rar[symbol={\rightsquigarrow}]
& \mathbf{1}^3 \mathbf{N}_2 \rar[symbol={\rightsquigarrow}]
& \mathbf{1}  \mathbf{N}_4 \rar[symbol={\rightsquigarrow}] \ar[dr,symbol={\rightsquigarrow}]
& \mathbf{N}_5   \rar[symbol={\rightsquigarrow}]
& \mathbf{1}^2 \mathbf{N}_3 \rar[symbol={\rightsquigarrow}] \ar[dr,symbol={\rightsquigarrow}]
& \mathbf{0} \mathbf{1}^4   \ar[dr,symbol={\rightsquigarrow}]
\\
&
&
& \mathbf{1}\mathbf{N}_2^2 \ar[rr,symbol={\rightsquigarrow}]
&
& \mathbf{N}_2 \mathbf{N}_3 \rar[symbol={\rightsquigarrow}]
& \mathbf{0}\mathbf{1}^2\mathbf{N}_2
\end{tikzcd}
\]
\[
\begin{array}{c|*{9}c}
\mathbf{b}
& \mathbf{1}^{5}
& \mathbf{1}^3\mathbf{N}_2
& \mathbf{1}\mathbf{N}_4
& \mathbf{N}_5
& \mathbf{1}^2\mathbf{N}_3
& \mathbf{1}  \mathbf{N}_2^{2}
& \mathbf{0}\mathbf{1}^{4}
& \mathbf{N}_2\mathbf{N}_3
& \mathbf{0}\mathbf{1}^2\mathbf{N}_2 \\
\hline
\dim\qbics_{V,\mathbf{b}}
& 25
& 24
& 23
& 22
& 21
& 21
& 20
& 20
& 19
\end{array}
\]
\caption{Immediate specialization relations amongst \(5\)-dimensional \(q\)-bic
forms and dimensions of the corresponding strata, up to the first few with
nontrivial radical.}
\label{moduli-five-dim-figure}
\end{figure}

\begin{IntroTheorem}\label{moduli-theorem}
Let \(\beta\) and \(\beta'\) be \(q\)-bic forms on a \(\kk\)-vector space \(V\).
\begin{enumerate}
\item\label{moduli-theorem.sufficient}
If the orbit closure of \(\beta\) contains \(\beta'\), then
\(\Psi_m(\beta) \leq \Psi_m(\beta')\) for all \(m \geq 1\).
\item\label{moduli-theorem.necessary}
If \(\Psi_m(\beta) \leq \Psi_m(\beta')\) and
\(\Theta_m(\beta) \leq \Theta_m(\beta')\) for all \(m \geq 1\), then the orbit
closure of \(\beta\) contains \(\beta'\).
\end{enumerate}
\end{IntroTheorem}

See \parref{moduli-specializations-necessary} and
\parref{moduli-specializations-sufficient}. Figure
\parref{moduli-five-dim-figure} illustrates the result for \(5\)-dimensional
\(q\)-bic forms. Clearly, this can be sharpened: for instance, the hypothesis
in \ref{moduli-theorem.necessary} can be relaxed. See
\parref{moduli-specializations-remarks} for further comments and questions
about the moduli of \(q\)-bic forms.

\subsection*{Applications and interrelations}\textbf{---}\;
Isotropic lines for a nonzero \(q\)-bic form \(\beta\) on \(V\) are parameterized
by a degree \(q+1\) hypersurface \(X\) in the projective space \(\mathbf{P}V\):
such is a \emph{\(q\)-bic hypersurface}. For example, if \(\beta\) admits
\(\mathbf{1}^{\oplus n}\) as a Gram matrix, then \(X\) is the Fermat
hypersurface of degree \(q+1\). As discussed already, smooth \(q\)-bic
hypersurfaces have long been of recurring interest for many reasons;
see also \cite[pp.7--11]{thesis} for further references and discussion. The
theory developed here provides powerful new methods to the study of these
objects.

There have been a few precedents to this circle of ideas in the literature:
Thinking of \(q\)-bic hypersurfaces in terms of a bilinear form is related to
the study of finite Hermitian varieties, as in \cite{BC:Hermitian,
Segre:Hermitian}. This perspective is also adopted by \cite{Shimada:Lattices}
in a spirit closer to that of the present article. The classification of
\(q\)-conjugacy classes of matrices has been considered in varying degrees of
generality: for nonsingular matrices, this follows from Lang's Lemma, see
\cite[Theorem 1]{Lang}, and can be found in \cite{Hefez:Thesis,
Beauville:Moduli}; for corank \(1\) matrices, this is \cite[Proposition
1]{HH:Fermat}; and for general matrices, this is \cite[Theorem
7.1]{KKPSSW:F-Pure}.

\subsection*{Outline}\textbf{---}\;
Basic definitions and constructions are given in \S\parref{section-definitions}.
The relationship between \(q\)-bic forms and Hermitian forms over
\(\mathbf{F}_{q^2}\) is explained in \S\parref{section-hermitian}; in
particular, \parref{forms-hermitian-diagonal} shows that any nonsingular
\(q\)-bic form over a separably closed field admits an orthonormal basis. The
\(\perp\)- and \(\perp^{[\smallbullet]}\)-filtrations are constructed in
\S\parref{section-filtrations}, and certain basic properties such as symmetry
and the meaning of the \(b_m\) are established: see
\parref{forms-intersect-filtrations} and \parref{forms-invariants-decreasing}.
Classification of \(q\)-bic forms by their numerical invariants is accomplished
in \S\parref{section-classification}. Automorphism group schemes of \(q\)-bic
forms are studied in \S\parref{section-automorphisms}. The article closes in
\S\parref{section-moduli} with a basic study of the moduli space of \(q\)-bic
forms.

\subsection*{Acknowledgements}\textbf{---}\;
This paper is based and expands on Chapter 1 of my thesis
\cite{thesis}. Much thanks goes to Aise Johan de Jong for sharing an interest
on this and related topics over the years. Thanks also to Jason Starr and
Bjorn Poonen for helpful comments. During the initial stages of this work, I
was partially supported by an NSERC Postgraduate Scholarship.

\section{Basic notions}\label{section-definitions}
This Section begins with the basic definitions, constructions, and properties
of \(q\)-bic forms. For the schematic constructions later, this material
is developed in the setting of finite projective modules over an
\(\mathbf{F}_{q^2}\)-algebra \(R\). To set notation: Write
\(\Fr \colon R \to R\) for the \(q\)-power Frobenius morphism.
Given an \(R\)-module \(M\), write \(M^\vee \coloneqq \Hom_R(M,R)\) for its
dual; the double dual is tacitly identified with \(M\)
itself. For each integer \(i \geq 0\), write
\(M^{[i]} \coloneqq R \otimes_{\Fr^i,R} M\) for the \emph{\(i\)-th Frobenius
twist} of \(M\): the (left) \(R\)-module on which the action of \(R\) is
twisted by the \(q^i\)-power Frobenius. The canonical map \(M \to M^{[i]}\)
given by \(m \mapsto m^{[i]} \coloneqq 1 \otimes m\) is the universal
\(\Fr^i\)-linear map out of \(M\). Given a submodule \(N' \subseteq
M^{[i]}\), a submodule \(N \subseteq M\) such that \(N^{[i]} = N'\) is called a
\emph{Frobenius descent} of \(N'\) to \(M\).

\subsectiondash{Definitions}\label{forms-definition}
A \emph{\(q\)-bic form} over \(R\) is a pair \((M,\beta)\) consisting of an
\(R\)-module \(M\) and an \(R\)-linear map
\(\beta \colon M^{[1]} \otimes_R M \to R\). A \emph{morphism}
\(\varphi \colon (M_1,\beta_1) \to (M_2,\beta_2)\)
between two \(q\)-bic forms is a morphism \(\varphi \colon M_1 \to M_2\) of
\(R\)-modules such that, for every \(m \in M_1\) and \(m' \in M_1^{[1]}\),
\[
\beta_1(m',m) = \beta_2(\varphi^{[1]}(m'), \varphi(m))
\]
where \(\varphi^{[1]} \colon M_1^{[1]} \to M_2^{[1]}\) is the Frobenius twist
of \(\varphi\). The morphism \(\varphi\) is an \emph{isomorphism} if the underlying
module map is an isomorphism.

Adjunction induces two mutually dual \(R\)-linear maps which, by an abuse of
notation, are denoted
\[
\beta \colon M \to M^{[1],\vee}
\quad\text{and}\quad
\beta^\vee \colon M^{[1]} \to M^\vee.
\]
The form \((M,\beta)\) is said to be \emph{nondegenerate} if
the map \(\beta \colon M \to M^{[1]}\) is injective, and
\emph{nonsingular} if this map is an isomorphism.

\subsectiondash{Gram matrices}\label{forms-gram-matrix}
Suppose that \(M\) is, moreover, a finite free \(R\)-module. The
\emph{Gram matrix} of a \(q\)-bic form \(\beta\) with respect to a basis
\(\varphi \colon \bigoplus_{i = 1}^n R \cdot e_i \xrightarrow{\sim} M\) is the
\(n\)-by-\(n\) matrix
\[
\Gram(\beta;e_1,\ldots,e_n) \coloneqq
\big(\beta(\varphi(e_i)^{[1]},\varphi(e_j))\big)_{i, j = 1}^n.
\]
Equivalently, this is the matrix of the map \(\beta \colon M \to M^{[1],\vee}\)
with respect to the bases
\(\varphi \colon \bigoplus_{i = 1}^n R \cdot e_i \xrightarrow{\sim} M\) and
\(\varphi^{[1]} \colon \bigoplus_{i = 1}^n R \cdot e_i^{[1]} \xrightarrow{\sim} M^{[1]}\).

Given another basis
\(\varphi' \colon \bigoplus_{i = 1}^n R \cdot e_i' \xrightarrow{\sim} M\),
view the change of basis isomorphism
\(A \coloneqq \varphi^{-1} \circ \varphi'\) as an invertible \(n\)-by-\(n\)
matrix over \(R\). Its Frobenius twist \(A^{[1]}\) is the matrix obtained from
\(A\) by taking \(q\)-powers entrywise, and the two Gram matrices are related
by
\[
\Gram(\beta; e_1',\ldots,e_n') =
A^{[1],\vee} \cdot \Gram(\beta; e_1,\ldots,e_n) \cdot A.
\]
It is now clear that a \(q\)-bic form \(\beta\) on a free module \(M\) is
nonsingular if and only if its Gram matrix in some---equivalently, any---basis
is invertible.

\subsectiondash{Standard forms}\label{forms-standard}
Gram matrices provide a convenient way to encode \(q\)-bic forms.
Given an \(n\)-by-\(n\) matrix \(B\) over \(R\), let \((R^{\oplus n}, B)\)
denote the unique \(q\)-bic form with Gram matrix \(B\) in the given basis.
Particularly simple, and important, are the following:
For each integer \(m \geq 1\), let
\[
\mathbf{N}_m \coloneqq
\left(
\begin{smallmatrix}
0 & 1 & \cdots &   &   \\
  & 0 & \cdots &   &   \\
  &   & \cdots &   &   \\
  &   & \cdots & 0 & 1 \\
  &   & \cdots &   & 0
\end{smallmatrix}
\right)
\]
denote the \(m\)-by-\(m\) Jordan block with \(0\) on the diagonal, and write
\(\mathbf{1}\) for the \(1\)-by-\(1\) identity matrix.
Given two matrices \(B_1\) and \(B_2\), write \(B_1 \oplus B_2\) for their
block diagonal sum. Let \(\mathbf{b} \coloneqq (a;b_1,b_2,\ldots)\) be a
sequence of nonnegative integers such that \(n = a + \sum_{m \geq 1} m b_m\).
The \(q\)-bic form
\[
\Big(
R^{\oplus n},
\mathbf{1}^{\oplus a} \oplus
\Big(\bigoplus\nolimits_{m \geq 1} \mathbf{N}_m^{\oplus b_m}\Big)
\Big)
\]
is the \emph{standard \(q\)-bic form of type \(\mathbf{b}\)}.

\subsectiondash{Orthogonals}\label{forms-orthogonal-inclusion-reversing}
To discuss orthogonals, consider generally a map
\(\beta \colon M_2 \otimes_R M_1 \to R\) of \(R\)-modules, viewed
as a pairing between \(M_1\) and \(M_2\). Given submodules
\(N_i \subseteq M_i\) for \(i = 1,2\), write
\[
N_1^\perp
\coloneqq \ker(M_2 \xrightarrow{\beta^\vee} M_1^\vee \to N_1^\vee)
\quad\text{and}\quad
N_2^\perp
\coloneqq \ker(M_1 \xrightarrow{\beta} M_2^\vee \to N_2^\vee).
\]
These are the \emph{orthogonals}, with respect to \(\beta\), of \(N_1\) and
\(N_2\), respectively. The orthogonals \(M_2^\perp \subseteq M_1\) and
\(M_1^\perp \subseteq M_2\) are called the \emph{kernels} of
\(\beta\) (and \(\beta^\vee\)). It is formal to see that taking orthogonals
is an inclusion-reversing operation. Precisely: for submodules \(N_1\) and
\(N_1'\) of \(M_1\):
\begin{enumerate}
\item\label{forms-orthogonal-inclusion-reversing.nested}
if \(N_1' \subseteq N_1\), then \(N_1'^\perp \supseteq N_1^\perp\);
\item\label{forms-orthogonal-inclusion-reversing.sum}
\((N_1 + N_1')^\perp = N_1^\perp \cap N_1'^\perp\); and
\item\label{forms-orthogonal-inclusion-reversing.intersection}
\((N_1 \cap N_1')^\perp \supseteq N_1^\perp + N_1'^\perp\).
\end{enumerate}
With further assumptions, orthogonals behave as expected:

\begin{Lemma}\label{forms-orthogonal-sequence}
Suppose that the image of the map \(\beta \colon M_1 \to M_2^\vee\)
is a local direct summand, and that \(N_1', N_1 \subseteq M_1\) are local
direct summands. Then there are exact sequences:
\begin{enumerate}
\item\label{forms-orthogonal-sequence.everything}
\(
  0 \to
  M_2^\perp \to
  M_1 \to
  M_2^\vee \to
  M_1^{\perp,\vee} \to
  0
\),
\item\label{forms-orthogonal-sequence.submodule}
\(
  0 \to
  N_1 \cap M_2^\perp \to
  N_1 \to
  M_2^\vee \to
  N_1^{\perp,\vee} \to
  0
\), and
\item\label{forms-orthogonal-sequence.nested}
\(
0 \to
N_1 \cap M_2^\perp/ N_1' \cap M_2^\perp \to
N_1/N_1' \to
(N_1'^\perp/N_1^\perp)^\vee \to
0
\) if additionally \(N_1' \subseteq N_1\).
\end{enumerate}
The operation of taking orthogonals further satisfy:
\begin{enumerate}
\setcounter{enumi}{3}
\item\label{forms-orthogonal-reflexive}
\(N_1^{\perp,\perp} =  M_2^\perp + N_1\), that is, reflexivity, and
\item\label{forms-orthogonal-cap-is-sum}
\((N_1 \cap N_1')^\perp = N_1^\perp + N_1'^\perp\)
if additionally \(M_2^\perp \subseteq N_1\).
\end{enumerate}
\end{Lemma}

\begin{proof}
For \ref{forms-orthogonal-sequence.everything}, it remains to identify
\(\coker(\beta)\) with \(M_1^{\perp,\vee}\) for the kernel-cokernel exact
sequence of \(\beta \colon M_1 \to M_2^\vee\). The hypotheses imply that the
cokernel is a local direct summand of \(M_2^\vee\), so that dualizing the
sequence yields an exact sequence
\[
0 \to
\coker(\beta)^\vee \to
M_2 \xrightarrow{\beta^\vee}
M_1^\vee \to
M_2^{\perp,\vee} \to
0.
\]
This identifies \(\coker(\beta)^\vee\) as \(\ker(\beta^\vee) \eqqcolon M_1^\perp\).
Reflexivity shows \(\coker(\beta) = M_1^{\perp,\vee}\). Sequence
\ref{forms-orthogonal-sequence.submodule} now follows upon restricting
\(\beta\) to \(N_1\). Comparing this sequence for nested submodules
\(N_1' \subseteq N_1\) then gives a commutative diagram of exact sequences
\[
\begin{tikzcd}[row sep=0.7em]
0 \rar
& N_1' \cap M_2^\perp \rar \dar[symbol={\subseteq}]
& N_1' \rar \dar[symbol={\subseteq}]
& (M_2/N_1'^\perp)^\vee \rar \dar[symbol={\subseteq}]
& 0 \\
0 \rar
& N_1 \cap M_2^\perp \rar
& N_1 \rar
& (M_2/N_1^\perp)^\vee \rar
& 0
\end{tikzcd}
\]
from which \ref{forms-orthogonal-sequence.nested} follows upon taking cokernels.

For \ref{forms-orthogonal-reflexive}, note that
\(
N_1^{\perp,\perp}
= \ker(\beta \colon M_1 \to M_2^\vee \to N_1^{\perp,\vee})
= M_2^\perp + \beta^{-1}(\ker(M_2^\vee \to N_1^{\perp,\vee}))
\),
so that the result follows from \ref{forms-orthogonal-sequence.submodule},
which shows that the second term is \(N_1\).

For \ref{forms-orthogonal-cap-is-sum}, apply
\parref{forms-orthogonal-inclusion-reversing}\ref{forms-orthogonal-inclusion-reversing.sum}
to \(N_1^\perp\) and \(N_1'^\perp\), take an orthogonal, apply
\ref{forms-orthogonal-reflexive}, and use the assumption \(M_2^\perp \subseteq N_1\)
to obtain the first equality in
\[
N_1^\perp + N_1'^\perp =
(N_1 \cap (N_1' + M_2^\perp))^\perp =
(N_1 \cap N_1' + M_2^\perp)^\perp =
(N_1 \cap N_1')^\perp
\]
The second equality comes from \parref{forms-orthogonal-intersection-sum} below,
and the final equality comes from
\parref{forms-orthogonal-inclusion-reversing}\ref{forms-orthogonal-inclusion-reversing.sum}
and \ref{forms-orthogonal-reflexive}.
\end{proof}

\begin{Lemma}\label{forms-orthogonal-intersection-sum}
Let \(M\) be an \(R\)-module and \(K, N, N' \subseteq M\) submodules. If
\(K \subseteq N\), then
\[
N \cap (N' + K) = N \cap N' + K.
\]
\end{Lemma}

\begin{proof}
The inclusion ``\(\supseteq\)'' always holds and is clear.
For ``\(\subseteq\)'', let \(x \in N \cap (N' + K)\) and view it as an element
of \(N\). Being in the intersection means there are
\(x' \in N'\) and \(y \in K\) such that \(x = x' + y\). But
\(x' = x - y \in N\) since \(K \subseteq N\), so \(x' \in N \cap N'\).
This shows \(x \in N \cap N' + K\), as required.
\end{proof}

\subsectiondash{Frobenius twists and orthogonals}\label{forms-fr-twist-orthogonal}
Return to the situation of a \(q\)-bic form \((M,\beta)\) over \(R\).
Twisting by Frobenius yields a sequence of associated \(q\)-bic forms
\((M^{[i]},\beta^{[i]})\) for each integer \(i \geq 0\); the \(i\)-th Frobenius
twisted form is characterized by the formula
\[
\beta^{[i]}(m'^{[i]}, m^{[i]}) = \beta(m',m)^{q^i}
\quad\text{for every}\;
m \in M \;\text{and}\; m' \in M^{[1]}.
\]
Observe that the module \(M^{[i]}\) is involved with the
two Frobenius twists of \(\beta\), namely \(\beta^{[i-1]}\) and
\(\beta^{[i]}\). Accordingly, a submodule \(N \subseteq M^{[i]}\) induces two
orthogonals:
\begin{align*}
N^{\perp^{[i-1]}}
& \coloneqq
\ker\big(\beta^{[i-1]} \colon
M^{[i-1]} \to M^{[i],\vee} \to N^\vee\big),\;\text{and} \\
N^{\perp^{[i]}}
& \coloneqq
\ker\big(\beta^{[i],\vee} \colon
M^{[i+1]} \to M^{[i],\vee} \to N^\vee\big).
\end{align*}
These are the \((i-1)\)-th and \(i\)-th \emph{Frobenius-twisted
orthogonals of \(N\)}.

In this setting, suppose further that taking Frobenius twists commutes with
taking kernels: for instance, suppose that \(R\) is regular so that the
Frobenius morphism is flat by \cite{Kunz}; or suppose that the image in
\(N^\vee\) of \(\beta^{[i-1]}\) or \(\beta^{[i],\vee}\) has vanishing
\(\Tor_1^R\). Then the Frobenius-twisted orthogonal commutes with taking
Frobenius twists in the following sense: for any integer \(j \geq 0\),
\[
N^{\perp^{[i-1]},[j]} = N^{[j],\perp^{[i+j-1]}}
\quad\text{and}\quad
N^{\perp^{[i]},[j]} = N^{[j],\perp^{[i+j]}}
\]
as submodules of \(M^{[i+j-1]}\) and \(M^{[i+j+1]}\), respectively.

\subsectiondash{Total orthogonals and complements}\label{forms-orthogonal-complements}
The \emph{total orthogonal} of a submodule \(N \subseteq M\) is
\[
N^\perp \cap N^{[2],\perp^{[1]}}
= \Set{m' \in M^{[1]}
| \beta(m',n) = \beta^{[1]}(n^{[2]}, m') = 0\;\text{for all}\; n \in N}.
\]
This is a submodule of \(M^{[1]}\) and, in general, it need not have a
Frobenius descent to \(M\). The \emph{radical} of \((M,\beta)\) is
the total orthogonal of \(M\) itself, and may be written as
\[
\operatorname{rad}(\beta) =
\Set{m' \in M^{[1]} |
\beta(m', m) = \beta^{[1]}(m^{[2]}, m') = 0\;
\text{for all}\; m \in M}.
\]
In particular, \(\beta^{[1]}\) passes to the quotient and induces a \(q\)-bic
form on \(M^{[1]}/\operatorname{rad}(\beta)\).

An \emph{orthogonal complement} to a submodule \(M'\) is a submodule
\(M''\) such that \(M = M' \oplus M''\) and \(M''^{[1]}\) lies in the
total orthogonal of \(M'\). This situation is signified by
\[ (M,\beta) = (M',\beta') \perp (M'',\beta'') \]
where \(\beta'\) and \(\beta''\) denote the restriction of \(\beta\) to
\(M'\) and \(M''\), respectively; when the underlying modules are understood,
also write \(\beta = \beta' \oplus \beta''\). Orthogonal complements need
not exist, and when they do, need not be unique: one exists if and only if
the total orthogonal contains a module complement that descends to \(M\); it is
unique if and only if the total orthogonal is the complement.

\section{Hermitian forms}\label{section-hermitian}
A \(q\)-bic form \((M,\beta)\) over \(R\) linearizes a biadditive map
\(M \times M \to R\) which is \(R\)-linear in the second
variable, but only \(q\)-linear in the first. In the case that
\(R = \mathbf{F}_{q^2}\), such forms are sesquilinear with respect the
\(q\)-power Frobenius, and the Hermitian property gives a notion of symmetry.
In general, the Hermitian condition does not make sense. This Section identifies
a distinguished subset of any \(q\)-bic form consisting of those elements
that satisfy the Hermitian equation. This construction provides an invariant of
\(q\)-bic forms, sensitive to the arithmetic of its ring of definition.

\subsectiondash{Hermitian elements}\label{forms-hermitian}
An element \(m \in M\) is said to be \emph{Hermitian} if
\[
\beta(-,m) =
\beta^{[1]}(m^{[2]},-)\;
\text{as elements of \(M^{[1],\vee}\)}
\]
or equivalently, if \(\beta(n^{[1]},m) = \beta(m^{[1]},n)^q\) for all \(n \in M\),
see \parref{forms-fr-twist-orthogonal}.
It is straightforward to check that the set \(M_{\mathrm{Herm}}\) of Hermitian
elements is a vector space over \(\mathbf{F}_{q^2}\) and that
\(\beta(m_1^{[1]}, m_2) \in \mathbf{F}_{q^2}\) for every
\(m_1,m_2 \in M_{\mathrm{Herm}}\). Therefore the restriction of \(\beta\)
to the space of Hermitian elements gives
\[
\beta_{\mathrm{Herm}} \colon
M^{[1]}_{\mathrm{Herm}} \otimes_{\mathbf{F}_{q^2}} M_{\mathrm{Herm}} \to
\mathbf{F}_{q^2}
\]
a Hermitian form for the quadratic extension
\(\mathbf{F}_{q^2}/\mathbf{F}_q\). This partially justifies the nomenclature:
if \(M\) is a free module with basis \(\langle e_1,\ldots,e_n \rangle\)
consisting of Hermitian elements, then the associated Gram matrix is a
Hermitian matrix over \(\mathbf{F}_{q^2}\) in that it satisfies
\[
\Gram(\beta;e_1,\ldots,e_n)^\vee =
\Gram(\beta;e_1,\ldots,e_n)^{[1]}.
\]

The space \(M_{\mathrm{Herm}}\) may not be of finite dimension over
\(\mathbf{F}_{q^2}\). For instance, it contains the preimage in \(M\) of the
radical of \(\beta\). It is therefore helpful to note that taking Hermitian
vectors is compatible with orthogonal decompositions:

\begin{Lemma}\label{forms-hermitian-basics-orthogonals}
An orthogonal decomposition
\((M,\beta) = (M',\beta') \perp (M'',\beta'')\)
induces an orthogonal decomposition of Hermitian spaces
\[
(M_{\mathrm{Herm}},\beta_{\mathrm{Herm}}) =
(M_{\mathrm{Herm}}',\beta_{\mathrm{Herm}}') \perp
(M_{\mathrm{Herm}}'',\beta_{\mathrm{Herm}}'').
\]
\end{Lemma}

\begin{proof}
Let \(m \in M_{\mathrm{Herm}}\), and let \(m = m' + m''\) be its
decomposition with \(m' \in M'\) and \(m'' \in M''\). Consider \(m'\).
Since \(M'\) and \(M''\) are orthogonal, for every \(n \in M'\),
\[
\beta'(n^{[1]}, m')
= \beta(n^{[1]}, m' + m'')
= \beta((m' + m'')^{[1]}, n)^q
= \beta'(m'^{[1]},n)^q.
\]
Therefore \(m' \in M_{\mathrm{Herm}}'\). An analogous
argument shows \(m'' \in M_{\mathrm{Herm}}''\).
\end{proof}

In the remainder of this Section, let
\((V,\beta)\) be a \(q\)-bic form over a field \(\kk\) containing \(\mathbf{F}_{q^2}\).

\subsectiondash{Scheme of Hermitan vectors}\label{forms-hermitian-equations}
The subgroup \(V_{\mathrm{Herm}}\) of Hermitian
vectors of \(V\) may be endowed with the structure of a closed subgroup scheme
of the affine space \(\mathbf{A} V \coloneqq \Spec\Sym V^\vee\) on \(V\).
Namely, the functor taking a \(\kk\)-algebra \(R\) to the group of Hermitian
elements of the \(q\)-bic form \((V \otimes_\kk R, \beta \otimes_\kk R)\) over
\(R\) obtained by extension of scalars is represented by the closed subscheme
\[
\mathbf{A}V_{\mathrm{Herm}} =
\mathrm{V}\big(\beta(\phantom{-},-) - \beta^{[1]}(-^{[2]},\phantom{-})\big)
\]
cut out by the difference of two morphisms
\(\mathbf{A} V \to \mathbf{A} V^{[1],\vee}\) which are induced by the linear maps
\[
v \mapsto
\beta(\phantom{-},v) \in \Hom_\kk(V^{[1]}, V^\vee)
\quad\text{and}\quad
v \mapsto
\beta^{[1]}(v^{[2]},\phantom{-}) \in \Hom_\kk(V^{[1]}, V^{\vee,[2]}),
\]
where \(V^{\vee,[2]}\) is identified as the subspace of
\(q^2\)-powers in \(\Sym^{q^2}(V^\vee) \subset \Gamma(\mathbf{A} V, \sO_{\mathbf{A} V})\).
Concretely, choose a basis \(V = \langle e_1,\ldots,e_n \rangle\),
let \(B \coloneqq \Gram(\beta;e_1,\ldots,e_n)\) be the associated Gram matrix,
and let \(\mathbf{x}^\vee \coloneqq (x_1,\ldots,x_n)\) be the associated
coordinates for \(\mathbf{A}V \cong \mathbf{A}^n\).
Then \(\mathbf{A}V_{\mathrm{Herm}}\) is the closed subscheme given by
the system of equations
\[ B \mathbf{x} - B^{[1],\vee} \mathbf{x}^{[2]} = 0. \]
Observe furthermore that, by \parref{forms-hermitian}, \(\mathbf{A}
V_{\mathrm{Herm}}\) is a scheme of vector spaces over \(\mathbf{F}_{q^2}\), and
is equipped with a Hermitian bilinear form
\[
\mathbf{A}\beta_{\mathrm{Herm}} \colon
\mathbf{A} V_{\mathrm{Herm}}^{[1]} \times_\kk
\mathbf{A} V_{\mathrm{Herm}} \to
\mathbf{F}_{q^2, \kk}
\]
with values in the constant group scheme associated with \(\mathbf{F}_{q^2}\).

\subsectiondash{Examples}\label{forms-hermitian-examples}
Schemes of Hermitian vectors for the standard forms of \parref{forms-standard}
are determined with the help of \parref{forms-hermitian-basics-orthogonals}
by the following two computations:

First, let \((V,\beta) = (\mathbf{k}^{\oplus n},\mathbf{N}_n)\) be
the standard form with just one nilpotent block. Then
\[
\mathbf{A}V_{\mathrm{Herm}} =
\Spec
\kk[x_1,\ldots,x_n]/(x_2,x_3 - x_1^{q^2},\ldots, x_n - x_{n-2}^{q^2},x_{n-1}^{q^2}).
\]
The structure of this scheme depends on the parity of \(n\):
\[
\frac{\kk[x_1,\ldots,x_n]}{(x_2,x_3 - x_1^{q^2},\ldots,x_n- x_{n-2}^{q^2},x_{n-1}^{q^2})}
\cong
\begin{dcases*}
\kk[x_1] & if \(n\) is odd, and \\
\kk[x_1]/(x_1^{q^{n}}) & if \(n\) is even.
\end{dcases*}
\]
Indeed, the first \(n-1\) equations imply that the even-indexed
variables vanish, and that \(x_{2k+1}\) is the \(q^{2k}\)-power
of \(x_1\). As for the final equation: when \(n\) is odd, this is implied by
the vanishing of the even-indexed variables, whereas when \(n\) is even, this
shows that \(x_{n-1}^{q^2} = x_1^{q^n} = 0\).

Second, let \(V = \langle e \rangle\) be a \(1\)-dimensional vector space
and let \(\beta(e^{[1]}, e) = \lambda\) for some \(\lambda \in \kk\). Then
\[
\mathbf{A}V_{\mathrm{Herm}} =
\Spec\kk[x]/(\lambda x - \lambda^q x^{q^2}).
\]
This is a form of the constant group scheme on \(\mathbf{F}_{q^2}\) that
splits with a \((q+1)\)-th root of \(\lambda\). By
\parref{forms-hermitian-basics-orthogonals}, this determines the structure the
scheme of Hermitian vectors for all diagonal \(q\)-bic forms. This also suggests
that the scheme is quite simple for all nonsingular \(q\)-bic forms. Indeed:

\begin{Proposition}\label{forms-hermitian-nondegenerate}
If \((V,\beta)\) is a nonsingular \(q\)-bic form, then \(\mathbf{A}V_{\mathrm{Herm}}\) is an
\'etale group scheme of degree \(q^2n\) over \(\kk\), geometrically isomorphic
to the constant group scheme associated with \(\mathbf{F}_{q^2}^{\oplus n}\).
\end{Proposition}

\begin{proof}
Since \(\beta\) is nonsingular, the equations for
\(\mathbf{A}V_{\mathrm{Herm}}\) in \parref{forms-hermitian-equations} may be
written as
\[ \mathbf{x}^{[2]} = B^{[1],\vee,-1} B \mathbf{x}. \]
This is a system of \(n\) equations in \(n\) variables, with Jacobian equations
the linear system given by \(B^{[1],\vee,-1} B\). Since \(\beta\) is
nonsingular, this is of full rank, and so \(\mathbf{A}V_{\mathrm{Herm}}\) is
\'etale of degree \(q^2n\) over \(\kk\).
Since the group of points \(\mathbf{A}V_{\mathrm{Herm}}(\kk^{\mathrm{sep}})\)
is a vector space over \(\mathbf{F}_{q^2}\) by \parref{forms-hermitian},
\(\mathbf{A} V_{\mathrm{Herm}}\) must be a form of the constant group scheme
associated with \(\mathbf{F}_{q^2}^{\oplus n}\), see \cite[2.16]{Milne:AlgGroups}.
\end{proof}

This implies that a nonsingular \(q\)-bic form is spanned by its Hermitian
vectors after a finite separable field extension. This is not true in general:
\parref{forms-hermitian-examples} shows that a form of type
\(\mathbf{N}_{2k}\) has no nonzero Hermitian vectors.

\begin{Proposition}\label{forms-hermitian-nondegenerate-span}
If \((V,\beta)\) is a nonsingular \(q\)-bic form over a separably closed field
\(\kk\), then the natural map
\(V_{\mathrm{Herm}} \otimes_{\mathbf{F}_{q^2}} \kk \to V\) is an isomorphism.
\end{Proposition}

\begin{proof}
The two \(\kk\)-vector spaces have the same dimension by
\parref{forms-hermitian-nondegenerate}, so it suffices to show that the map
is injective. If not, there is a linear relation in \(V\) of the form
\[
v_{m+1} = a_1 v_1 + \cdots + a_m v_m
\quad\text{with}\; m \geq 1,\; v_i \in V_{\mathrm{Herm}},\;\text{and}\; a_i \in \kk.
\]
Choose such a relation with \(m\) minimal. Minimality implies that
\(v_1,\ldots,v_m\) are linearly independent in \(V\). Since \(\beta\) is
nonsingular, there exists \(w \in V\) such that \(\beta(w^{[1]},v_i) = 0\) for
\(1 \leq i \leq m-1\), and \(\beta(w^{[1]},v_m) \neq 0\); up to scaling
\(v_m\), this last value may be taken to be \(1\). Since the \(v_i\)
are Hermitian,
\[
a_m
= \beta(w^{[1]}, v_{m+1})
= \beta(v_{m+1}^{[1]},w)^q \\
= \sum\nolimits_{i = 1}^m a_i^{q^2} \beta(v_i^{[1]},w)^q
= \sum\nolimits_{i = 1}^m a_i^{q^2} \beta(w^{[1]},v_i)
= a_m^{q^2}
\]
so \(a_m \in \mathbf{F}_{q^2}\), whence
\(v_m' \coloneqq v_{m+1} - a_m v_m\) lies in \(V_{\mathrm{Herm}}\) by
\parref{forms-hermitian}. The relation
\(v_m' = a_1 v_1 + \cdots + a_{m-1} v_{m-1}\) then
contradicts the minimality of \(m\), as required.
\end{proof}

In particular, since Hermitian forms over \(\mathbf{F}_{q^2}\) always have
an orthonormal basis, this implies that all nonsingular \(q\)-bic forms over
a separably closed field have a Gram matrix given by the identity:

\begin{Corollary}\label{forms-hermitian-diagonal}
If \((V,\beta)\) is a nonsingular \(q\)-bic form over a separably closed field,
then there exists a basis \(V = \langle e_1,\ldots,e_n \rangle\) such that
\(\Gram(\beta; e_1,\ldots,e_n) = \mathbf{1}^{\oplus n}\).
\qed
\end{Corollary}

When \(\kk\) is separably closed, \(V\) is spanned by its Hermitian vectors
\(V_{\mathrm{Herm}}\) by \parref{forms-hermitian-nondegenerate-span}, and so
the restriction map \(\Aut(V,\beta) \to
\Aut(V_{\mathrm{Herm}},\beta_{\mathrm{Herm}})\) from the automorphism group of
a nonsingular \(q\)-bic form to that of its associated Hermitian form is an
isomorphism. Combined with \parref{forms-hermitian-diagonal}, Galois descent
implies that isomorphism classes of nonsingular \(q\)-bic forms of dimension
\(n\) over a general field \(\kk\) are in bijection with forms of the standard
Hermitian form on \(\mathbf{F}_{q^2}^{\oplus n}\) over \(\kk\). This implies:

\begin{Corollary}\label{hermitian-galois}
Two nonsingular \(q\)-bic forms are isomorphic over \(\kk\) if and only
if their group schemes of Hermitian vectors equipped with their
Hermitian forms are isomorphic over \(\kk\).
\qed
\end{Corollary}

\section{Canonical filtrations}\label{section-filtrations}
Although a \(q\)-bic form \(\beta\) pairs distinct modules \(M\) and \(M^{[1]}\),
the canonical \(q\)-linear map \(M \to M^{[1]}\) makes it possible to iterate
left and right orthogonals. This gives rise to two intrinsic filtrations whose
interactions ultimately encode the structure of \(\beta\). After their
definition in \parref{forms-canonical-filtration} and
\parref{forms-canonical-filtration-second}, a symmetry relation is established
in \parref{forms-intersect-filtrations}, and a series of numerical invariants
of \(\beta\) are defined in \parref{forms-numerical-invariants} and
\parref{forms-numerical-invariants-type}.

\subsectiondash{\(\perp\)-filtration}\label{forms-canonical-filtration}
The map \(\beta \colon M \to M^{[1],\vee}\) gives rise to a sequence
\(\mathrm{P}_{\smallbullet} M\) of submodules of \(M\) as follows: set
\(\mathrm{P}_{-1}M \coloneqq \{0\}\), and for each \(i
\geq 0\), inductively set
\[
\mathrm{P}_i M
\coloneqq \mathrm{P}_{i-1} M^{[1],\perp}
\coloneqq \ker\big(\beta \colon M \to M^{[1],\vee} \to \mathrm{P}_{i-1} M^{[1],\vee}\big).
\]
For instance, \(\mathrm{P}_0 M = M\) is the entire module and
\(\mathrm{P}_1 M = M^{[1],\perp}\) is the kernel of \(\beta\).
Induction with the inclusion-reversing property
\parref{forms-orthogonal-inclusion-reversing}\ref{forms-orthogonal-inclusion-reversing.nested}
implies that:
\begin{itemize}
\item the odd-indexed submodules form an increasing filtration
\(\mathrm{P}_{2k-1}M \subseteq \mathrm{P}_{2k+1} M\),
\item
the even-indexed submodules form a decreasing filtration
\(\mathrm{P}_{2k} M \supseteq \mathrm{P}_{2k+2}M\), and
\item
the odd-indexed submodules are totally isotropic
\(\mathrm{P}_{2k-1}M \subseteq \mathrm{P}_{2k-1} M^{[1],\perp} \eqqcolon \mathrm{P}_{2k}M\).
\end{itemize}
Therefore the sequence \(\mathrm{P}_{\smallbullet} M\) fits into two interwoven
filtrations
\[
\{0\} =
\mathrm{P}_{-1} M \subseteq
\mathrm{P}_1 M \subseteq
\mathrm{P}_3 M \subseteq
\cdots \subseteq
\mathrm{P}_- M \subseteq
\mathrm{P}_+ M \subseteq
\cdots \subseteq
\mathrm{P}_4 M \subseteq
\mathrm{P}_2 M \subseteq
\mathrm{P}_0 M =
M
\]
called the \emph{\(\perp\)-filtration} of \((M,\beta)\); here,
\(\mathrm{P}_- M\) and \(\mathrm{P}_+ M\) are the limiting submodules for
the increasing odd-, and decreasing even-filtrations, respectively.

\subsectiondash{\(\perp^{[\smallbullet]}\)-filtration}\label{forms-canonical-filtration-second}
The map \(\beta^\vee \colon M^{[1]} \to M^\vee\) and its
Frobenius twists give rise to a sequence of submodules
\(\mathrm{P}'_i M^{[i]} \subseteq M^{[i]}\) as follows:
set \(\mathrm{P}'_{-1} M^{[-1]} \coloneqq \{0\}\), and for
each \(i \geq 0\), inductively set
\[
\mathrm{P}'_i M^{[i]} \coloneqq
\mathrm{P}'_{i-1}M^{[i-1],\perp^{[i-1]}} \coloneqq
\ker\big(\beta^{[i-1],\vee} \colon
M^{[i]} \to
M^{[i-1],\vee} \to
\mathrm{P}'_{i-1} M^{[i-1],\vee}\big),
\]
where notation is as in \parref{forms-fr-twist-orthogonal}.
For instance, \(\mathrm{P}'_0 M \coloneqq \mathrm{P}'_0 M^{[0]} = M\) is
the module itself and \(\mathrm{P}_1'M^{[1]} = M^\perp\) is the kernel of
\(\beta^\vee\). For each integer \(j\), write
\[
\mathrm{P}_i' M^{[i+j]} \coloneqq
(\mathrm{P}_i' M^{[i]})^{[j]}
\]
for the submodule of \(M^{[i+j]}\) which for \(j \geq 0\)
is the \(j\)-th Frobenius twist of \(\mathrm{P}_i' M^{[i]}\), and for
\(j \leq 0\) is the \(j\)-th Frobenius descent, if it exists; in the latter
case, say that the \(i\)-th piece of the
\(\perp^{[\smallbullet]}\)-filtration \emph{descends to \(M^{[i+j]}\) over \(R\)}.
As with the \(\perp\)-filtration, the inclusion-reversing property
\parref{forms-orthogonal-inclusion-reversing}\ref{forms-orthogonal-inclusion-reversing.nested}
inductively implies that the modules fit together to
yield interwoven filtrations, so that for each
\(i \geq 0\), there is a filtration of \(M^{[i]}\) of the form:
\[
\{0\} =
\mathrm{P}_{-1}' M^{[i]} \subseteq
\mathrm{P}_1'M^{[i]} \subseteq
\mathrm{P}_3'M^{[i]} \subseteq
\cdots
\subseteq
\mathrm{P}_i' M^{[i]} \subseteq
\cdots \subseteq
\mathrm{P}_2' M^{[i]} \subseteq
\mathrm{P}_0' M^{[i]} = M^{[i]}.
\]

Since the two filtrations are inductively defined via kernels of \(\beta\)
and \(\beta^\vee\), it is straightforward to check that their formation is
compatible with orthogonal decompositions:

\begin{Lemma}\label{forms-filtrations-sums}
If \((M,\beta) = (M',\beta') \perp (M'',\beta'')\) is an orthogonal decomposition,
then for every \(i\),
\[
\pushQED{\qed}
\mathrm{P}_i M =
\mathrm{P}_i M' \oplus \mathrm{P}_i M''
\quad\text{and}\quad
\mathrm{P}_i' M^{[i]} =
\mathrm{P}_i' M'^{[i]} \oplus
\mathrm{P}_i' M''^{[i]}.
\qedhere
\popQED
\]
\end{Lemma}

\subsectiondash{Example}\label{standard-forms-filtration-example}
Consider the \(\perp\)- and \(\perp^{[\smallbullet]}\)-filtrations for the
standard forms defined in \parref{forms-standard}. Note that the filtrations
are trivial when \((M,\beta)\) is nondegenerate, since they are constructed by
taking iterated kernels of \(\beta\) and \(\beta^\vee\). So by
\parref{forms-filtrations-sums}, it remains to describe
the filtrations when \(\beta\) has a Gram matrix given by \(\mathbf{N}_n\)
for some basis \(\langle e_1,\ldots,e_n \rangle\) of \(M\). A direct
computation shows that the \(\perp\)-filtration has \(n\) steps, given for
\(0 \leq i \leq n-1\) by
\[
\mathrm{P}_i M =
\begin{dcases*}
\phantom{\Big(}\bigoplus\nolimits_{\ell = 1}^k R \cdot e_{2\ell-1}
& if \(i = 2k - 1\), and \\
\Big(\bigoplus\nolimits_{\ell = 1}^k R \cdot e_{2\ell-1}\Big) \oplus
\Big(\bigoplus\nolimits_{j = 2k+1}^n R \cdot e_j\Big)
& if \(i = 2k\).
\end{dcases*}
\]
In particular, \(\mathrm{P}_- M = \mathrm{P}_+ M\) is the span of the odd-indexed
basis elements.

The \(\perp^{[\smallbullet]}\)-filtration similarly has \(n\) steps. Moreover,
in this special case, each step of the filtration descends to \(M\), and is
given for \(0 \leq i \leq n-1\) by
\[
\mathrm{P}_i' M =
\begin{dcases*}
\phantom{\Big(}
\bigoplus\nolimits_{\ell = 1}^k R \cdot e_{n-2\ell+2}
& if \(i = 2k-1\), and \\
\Big(\bigoplus\nolimits_{\ell = 1}^k R \cdot e_{n-2\ell+2}\Big) \oplus
\Big(\bigoplus\nolimits_{j = 1}^{n-2k} R \cdot e_j\Big)
& if \(i = 2k\).
\end{dcases*}
\]

For the remainder of the Section, specialize to the situation of a
\(q\)-bic form \((V,\beta)\) on a finite-dimensional vector space over a
field \(\kk\).

\subsectiondash{Symmetry}\label{forms-dimension-symmetry}
Dimensions of various pieces of the \(\perp\)- and
\(\perp^{[\smallbullet]}\)-filtration provide a series of numerical
invariants for \((V,\beta)\); taking just the first pieces gives the familiar
\begin{align*}
\rank(V,\beta)
& \coloneqq \rank(\beta \colon V \to V^{[1],\vee})
          = \rank(\beta^\vee \colon V^{[1]} \to V^\vee),\;\text{and} \\
\corank(V,\beta)
& \coloneqq \dim_\kk V - \rank(V,\beta)
          = \dim_\kk \mathrm{P}_1 V
          = \dim_\kk \mathrm{P}'_1 V^{[1]}.
\end{align*}
In particular, the first pieces of the two filtrations have the same
dimension. As is also suggested by the examples in
\parref{standard-forms-filtration-example}, this dimensional symmetry persists
amongst higher pieces. To prove this, first observe that
restricting the \(j\)-th Frobenius twisted pairing
\(\beta^{[j]} \colon V^{[j+1]} \otimes_{\kk} V^{[j]} \to \kk\)
to \(\mathrm{P}_{i-1} V^{[j+1]}\) and \(\mathrm{P}_j' V^{[j]}\) and using
\parref{forms-orthogonal-sequence}\ref{forms-orthogonal-sequence.submodule}
twice gives an exact sequence
\[
0 \to
\mathrm{P}_i V^{[j]} \cap \mathrm{P}_j' V^{[j]} \to
\mathrm{P}_j' V^{[j]} \xrightarrow{\beta^{[j]}}
\mathrm{P}_{i-1} V^{[j+1],\vee} \to
(\mathrm{P}_{i-1}V^{[j+1]} \cap \mathrm{P}_{j+1}' V^{[j+1]})^\vee \to
0.
\]
The symmetry statement is the following:

\begin{Proposition}\label{forms-intersect-filtrations}
\(
\dim_\kk\mathrm{P}_i V^{[j]} \cap \mathrm{P}_j' V^{[j]}
=
\dim_\kk\mathrm{P}_j V^{[i]} \cap \mathrm{P}_i' V^{[i]}
\)
for all integers \(i, j \geq 0\).
\end{Proposition}

\begin{proof}
By symmetry, it suffices to consider \(i \geq j \geq 0\). Proceed by
induction on \(i + j\). The base case is when \(i = j = i + j = 0\) so that
both sides are just \(V\) and there is nothing to prove. Now fix the quantity
\(i + j \geq 1\) and further induct on the difference
\(\delta \coloneqq i - j \geq 0\). There are
two base cases: If \(\delta = 0\), then there is nothing to prove. If
\(\delta = 1\), considering the sequence in
\parref{forms-dimension-symmetry} with \(i = j+1\) gives
\[
\dim_\kk \mathrm{P}_{j+1} V^{[j]} \cap \mathrm{P}_j' V^{[j]} -
\dim_\kk \mathrm{P}_j' V^{[j]} =
\dim_\kk \mathrm{P}_j V^{[j+1]} \cap \mathrm{P}_{j+1}' V^{[j+1]} -
\dim_\kk \mathrm{P}_j V^{[j+1]}.
\]
Since \(j < i + j = 2j+1\), induction applies to the negative terms to show
\[
\dim_\kk \mathrm{P}_j' V^{[j]} =
\dim_\kk \mathrm{P}_0 V^{[j]} \cap \mathrm{P}_j' V^{[j]} =
\dim_\kk \mathrm{P}_j V \cap \mathrm{P}_0' V =
\dim_\kk \mathrm{P}_j V
\]
whence
\(
\dim_\kk \mathrm{P}_{j+1} V^{[j]} \cap \mathrm{P}_j' V^{[j]} =
\dim_\kk \mathrm{P}_j V^{[j+1]} \cap \mathrm{P}_{j+1}' V^{[j+1]}
\).

Assume \(\delta \geq 2\). Taking dimensions in the sequence of
\parref{forms-dimension-symmetry} gives the first equation in:
\begin{align*}
\dim_\kk \mathrm{P}_i V^{[j]} \cap \mathrm{P}'_j V^{[j]}
& =
\dim_\kk \mathrm{P}_{i-1} V^{[j+1]} \cap \mathrm{P}'_{j+1} V^{[j+1]} -
\dim_\kk \mathrm{P}_{i-1} V^{[j+1]} +
\dim_\kk \mathrm{P}_j' V^{[j]} \\
& =
\dim_\kk \mathrm{P}_{j+1} V^{[i-1]} \cap \mathrm{P}'_{i-1} V^{[i-1]} -
\dim_\kk \mathrm{P}_{i-1}' V^{[i-1]} +
\dim_\kk \mathrm{P}_j V^{[i]} \\
& =
\dim_\kk \mathrm{P}_j V^{[i]} \cap \mathrm{P}_i' V^{[i]}.
\end{align*}
Since \(i-1 - (j+1) = \delta - 2\) and \(\max\{i-1,j\} < i+j\), induction gives
the equality in the middle. The final equality follows from the sequence in
\parref{forms-dimension-symmetry} upon substituting \(i \mapsto j+1\)
and \(j \mapsto i-1\).
\end{proof}

\subsectiondash{Numerical invariants}\label{forms-numerical-invariants}
A sequence of numerical invariants of \((V,\beta)\) is now obtained by taking
dimensions of graded pieces for the filtrations: For each integer \(m \geq 0\),
set \(\epsilon \coloneqq (-1)^m\), set
\[
a_m(V,\beta)
\coloneqq \dim_\kk \mathrm{P}_{m-\epsilon-1} V/\mathrm{P}_{m+\epsilon-1} V
= \dim_\kk \mathrm{P}_{m-\epsilon-1}' V^{[i]}/\mathrm{P}_{m+\epsilon-1}' V^{[m]},
\]
and set \(a(V,\beta) \coloneqq \dim_\kk \mathrm{P}_+ V/\mathrm{P}_- V\). Since
\(V\) is finite-dimensional, \(a_m(V,\beta)\) vanishes for large \(m\).
Slightly more convenient are the successive differences of these dimensions:
for each \(m \geq 0\), set
\[
b_m(V,\beta) \coloneqq a_m(V,\beta) - a_{m+1}(V,\beta)
\quad\text{so that}\quad
a_m(V,\beta) = \sum\nolimits_{i \geq m} b_i(V,\beta).
\]
In fact, these differences are also dimensions:

\begin{Lemma}\label{forms-invariants-decreasing}
\(\displaystyle
b_m(V,\beta) =
\dim_\kk\Bigg(
  \frac{\mathrm{P}_{m-\epsilon-1} V^{[1]} \cap \mathrm{P}_1' V^{[1]}}
  {\mathrm{P}_{m+\epsilon-1} V^{[1]} \cap \mathrm{P}_1' V^{[1]}}
\Bigg) =
\dim_\kk
\Bigg(
  \frac{\mathrm{P}_1 V^{[m]} \cap \mathrm{P}_{m-\epsilon-1}' V^{[m]}}
  {\mathrm{P}_1 V^{[m]} \cap \mathrm{P}_{m+\epsilon-1}' V^{[m]}}
\Bigg)
\).
\end{Lemma}

\begin{proof}
This follows directly from the definitions in \parref{forms-canonical-filtration}
and \parref{forms-canonical-filtration-second} of the filtrations, and taking
dimensions of the exact sequence
\parref{forms-orthogonal-sequence}\ref{forms-orthogonal-sequence.nested}
applied to the nested sequences of subspaces
\[
\mathrm{P}_{m+\epsilon-1} V^{[1]} \subseteq
\mathrm{P}_{m-\epsilon-1} V^{[1]} \subseteq V^{[1]}
\quad\text{and}\quad
\mathrm{P}_{m+\epsilon-1}' V^{[m]} \subseteq
\mathrm{P}_{m-\epsilon-1}' V^{[m]} \subseteq V^{[m]}
\]
for the maps \(\beta^\vee\) and \(\beta^{[m]}\), respectively.
\end{proof}

\subsectiondash{Type}\label{forms-numerical-invariants-type}
The sequence \((a(V,\beta); b_m(V,\beta))_{m \geq 1}\) is the \emph{type} of
\((V,\beta)\) and is the fundamental invariant: all other intrinsic numerical
invariants may be expressed in terms of the type. For instance,
\begin{align*}
\corank(V,\beta)
& = a_1(V,\beta) = \sum\nolimits_{m \geq 1} b_m(V,\beta), \;\;\text{and} \\
\dim_\kk V
& = a(V,\beta) + \sum\nolimits_{m \geq 1} a_m(V,\beta)
= a(V,\beta) + \sum\nolimits_{m \geq 1} mb_m(V,\beta).
\end{align*}
Dimensions of each intersection
\(\mathrm{P}_i V^{[j]} \cap \mathrm{P}_j' V^{[j]}\) may be expressed
similarly. Particularly useful is the following direct consequence of
\parref{forms-invariants-decreasing}: for every integer \(k \geq 1\),
\begin{align*}
\dim_\kk \mathrm{P}_{2k-1} V^{[1]} \cap \mathrm{P}_1' V^{[1]}
& = \sum\nolimits_{\ell = 1}^k b_{2\ell-1}(V,\beta),\;\;\text{and} \\
\dim_\kk \mathrm{P}_{2k} V^{[1]} \cap \mathrm{P}_1' V^{[1]}
& = \sum\nolimits_{\ell = 1}^k b_{2\ell-1}(V,\beta) + \sum\nolimits_{m \geq 2k+1} b_m(V,\beta).
\end{align*}

Examining the description of the filtrations for standard forms given in
\parref{standard-forms-filtration-example} shows that this notion of type
matches that given for standard forms in \parref{forms-standard}.
In particular, \(a(V,\beta)\) is the dimension of the nonsingular summand
of a standard form. This is true more generally in the following sense:

\begin{Lemma}\label{forms-nondegenerate-limit}
The restriction of \(\beta\) to \(\mathrm{P}_+ V\) has
\(\mathrm{P}_- V\) as its radical, and the induced \(q\)-bic form on
\(\mathrm{P}_+ V/\mathrm{P}_- V\) is nonsingular.
\end{Lemma}

\begin{proof}
It follows from definitions that
\(\mathrm{P}_{\pm} V = \mathrm{P}_{\mp} V^{[1],\perp}\), so the
kernels of \(\beta\) restricted to \(\mathrm{P}_+ V\) are
\begin{align*}
\ker(\beta \colon \mathrm{P}_+ V \to \mathrm{P}_+ V^{[1],\vee})
& = \mathrm{P}_+ V \cap \mathrm{P}_+ V^{[1],\perp}
= \mathrm{P}_- V, \;\;\text{and} \\
\ker(\beta^\vee \colon \mathrm{P}_+ V^{[1]} \to \mathrm{P}_+ V^\vee)
& = \mathrm{P}_+ V^{[1]} \cap \mathrm{P}_+ V^\perp
= \mathrm{P}_+ V^{[1]} \cap (\mathrm{P}_- V^{[1]} + \mathrm{P}_1' V^{[1]}),
\end{align*}
where the final equality follows from
\parref{forms-orthogonal-sequence}\ref{forms-orthogonal-reflexive}.
Since the two kernels have the same dimension, the latter must simply be
\(\mathrm{P}_- V^{[1]}\). Thus \(\beta\) restricted to \(\mathrm{P}_+ V\) has
\(\mathrm{P}_- V\) as its radical, and since this is the entire kernel,
the form induced on the quotient is nonsingular.
\end{proof}

\subsectiondash{Descending the \(\perp^{[\smallbullet]}\)-filtration}\label{filtrations-nu}
Since \(V\) is finite-dimensional, the filtrations are finite, and
all pieces of the \(\perp^{[\smallbullet]}\)-filtration are defined on all
sufficiently large Frobenius twists of \(V\). Thus
\[
\nu(V,\beta) \coloneqq
\min\Set{i \in \mathbf{Z}_{\geq 0} |
\text{\(\perp^{[\smallbullet]}\)-filtration descends to}\;
(V^{[i]},\beta^{[i]})\;
\text{over \(\kk\)}
}
\]
is well-defined. This depends on \(\kk\): for instance,
\(\nu(V,\beta) = 0\) for any form over a perfect field. The following
gives an \emph{a priori} upper bound \(\nu(V,\beta) \leq \nu_0\) depending only
on the type of \(\beta\). The statement is optimal, as can be shown by considering
infinitesimals in automorphism groups as in \parref{aut-kill-infinitesimals}
and comparing with the examples \parref{auts-examples-standard},
\parref{auts-examples-1a+N2b}, and
\(\AutSch_{(\kk^{\oplus 5}, \mathbf{N}_2\oplus\mathbf{N}_3)}\).

\begin{Lemma}\label{forms-frobenius-descent-index}
Assume that \(\beta\) is degenerate and let
\(\mu \coloneqq \max\set{m \in \mathbf{Z}_{\geq 1} | b_m(V,\beta) \neq 0}\). Then the
\(\perp^{[\smallbullet]}\)-filtration canonically descends to \(V^{[\nu_0]}\)
for
\[
\nu_0 \coloneqq
\begin{dcases*}
\mu - 2 & if \(\mu > 1\) and \(b_m(V,\beta) = 0\) for all even \(m\), \\
\mu - 1 & if \(\mu\) is odd, or \(\mu\) is even and \(a(V,\beta) = 0\), and \\
\mu     & otherwise.
\end{dcases*}
\]
\end{Lemma}

\begin{proof}
Since \(\mu\) is the length of the \(\perp^{[\smallbullet]}\)-filtration, the
last case is clear. When \(\mu\) is even and \(a(V,\beta) = 0\), then
\(\mathrm{P}_\mu' V^{[\mu]} = \mathrm{P}_{\mu-1}' V^{[\mu]}\), so the last step of the
filtration canonically descends to \(V^{[\mu-1]}\). It remains to consider the
situation when \(\mu\) is odd. Since \(b_m(V,\beta) = 0\) for all \(m > \mu\) the
formulae of \parref{forms-numerical-invariants} imply that
\(a_\mu(V,\beta) = b_\mu(V,\beta)\),  so that by comparing dimensions via
\parref{forms-invariants-decreasing}, the natural injection
\[
\mathrm{P}_1 V^{[\mu]} \cap \mathrm{P}_\mu' V^{[\mu]}/
\mathrm{P}_1 V^{[\mu]} \cap \mathrm{P}_{\mu-2}' V^{[\mu]}
\hookrightarrow
\mathrm{P}_\mu' V^{[\mu]}/
\mathrm{P}_{\mu-2}' V^{[\mu]}
\]
is an isomorphism. This implies that \(\mathrm{P}_\mu' V^{[\mu]}\) is spanned by
\(\mathrm{P}_1 V^{[\mu]} \cap \mathrm{P}_\mu' V^{[\mu]}\) and
\(\mathrm{P}_{\mu-2}' V^{[\mu-2]}\). By the formulae in
\parref{forms-numerical-invariants-type}, the intersection coincides with
\(\mathrm{P}_1 V^{[\mu]} \cap \mathrm{P}_{\mu-1}' V^{[\mu]}\). Therefore
\[
\mathrm{P}_\mu' V^{[\mu-1]} =
\mathrm{P}_{\mu-2}' V^{[\mu-1]} +
\mathrm{P}_1 V^{[\mu-1]} \cap \mathrm{P}_{\mu-1}' V^{[\mu-1]}
\]
is a canonical descent of \(\mathrm{P}_\mu' V^{[\mu]}\) to \(V^{[\mu-1]}\).
Furthermore, when \(\mu > 1\) and \(b_m(V,\beta) = 0\) for all even \(m\), the
intersection is all of \(\mathrm{P}_1 V^{[\mu-1]}\) and so the sum
descends further to \(V^{[\mu-2]}\). By
\parref{forms-orthogonal-sequence}\ref{forms-orthogonal-reflexive},
\[
\mathrm{P}_{\mu-1}' V^{[\mu-3]} \coloneqq
\mathrm{P}_\mu' V^{[\mu-2],\perp^{[\mu-3]}} \coloneqq
\ker\big(\beta^{[\mu-3]} \colon V^{[\mu-3]} \to \mathrm{P}_\mu' V^{[\mu-2],\vee}\big)
\]
and so the entire \(\perp^{[\smallbullet]}\)-filtration admits a descent
to \(V^{[\mu-2]}\) in this case.
\end{proof}

\section{Classification}\label{section-classification}
The object of this Section is to prove the following Classification Theorem,
which says that, after passing to a suitable Frobenius twist upon which its
\(\perp^{[\smallbullet]}\)-filtration is defined as in \parref{filtrations-nu},
a \(q\)-bic form over a field is essentially classified by its type, as defined
in \parref{forms-numerical-invariants-type}:

\begin{Theorem}\label{forms-classification-theorem}
Let \((V,\beta)\) be a \(q\)-bic form over a field \(\kk\) of
type \((a;b_m)_{m \geq 1}\), and let \(\nu \coloneqq \nu(V,\beta)\).
Then there exists an orthogonal decomposition
\[
\beta^{[\nu]} = \beta_0 \oplus \Big(\bigoplus\nolimits_{m \geq 1} \beta_m\Big)
\]
such that \(\beta_0\) is nonsingular of dimension \(a\), and \(\beta_m\)
has a Gram matrix given by \(\mathbf{N}_m^{\oplus b_m}\) for each \(m \geq 1\).
\end{Theorem}

Combined with the classification of nonsingular forms in
\parref{forms-hermitian-diagonal} and the remarks on \(\nu\) from
\parref{filtrations-nu}, this gives a normal form for \(q\)-bic forms over an
algebraically closed field:

\begin{Corollary}\label{forms-classification-theorem-kbar}
If \(\kk\) is algebraically closed, then there exists a basis
\(V = \langle e_1,\ldots,e_n \rangle\) such that
\[
\pushQED{\qed}
\Gram(\beta;e_1,\ldots,e_n) =
\mathbf{1}^{\oplus a} \oplus
\Big(\bigoplus\nolimits_{m \geq 1} \mathbf{N}_m^{\oplus b_m}\Big).
\qedhere
\popQED
\]
\end{Corollary}

Throughout this Section, let \((V,\beta)\) denote a \(q\)-bic form over a field
\(\kk\). The idea of proof is as follows: Let \(m \geq 1\) be an integer, and
choose a subspace of dimension \(b \coloneqq b_m(V,\beta)\) in
\[
\begin{dcases*}
\mathrm{P}_1 V^{[\nu]} \cap \mathrm{P}_{2k-1}' V^{[\nu]}
& if \(m = 2k-1\), or \\
\mathrm{P}_1 V^{[\nu]} \cap \mathrm{P}_{2k-2}' V^{[\nu]}
& if \(m = 2k\),
\end{dcases*}
\]
lifting the quotient from \parref{forms-invariants-decreasing}.
The seesaw relation between the two filtrations, as in
\parref{classification-propagate-subspace}, extends this subspace to an
\(mb\)-dimensional subspace of \(V\) on which \(\beta\) has a Gram matrix given
by \(\mathbf{N}_m^{\oplus b}\), and which admits an orthogonal complement.
Since the filtrations are compatible with orthogonal decompositions,
inductively continuing this method on the complement yields the Theorem.

\subsectiondash{Recognizing a standard form}\label{classification-standard-form}
An intrinsic formulation of the property that a \(q\)-bic form \((V,\beta)\)
has a Gram matrix given by \(\mathbf{N}_m^{\oplus b}\) is: there exists a
vector space decomposition \(V = \bigoplus_{i = 1}^m V_i\) such that, for
integers \(1 \leq i, j, \leq m\), the map
\[
\beta_{i,j} \colon
V_j \subseteq
V \xrightarrow{\beta}
V^{[1],\vee} \twoheadrightarrow
V_i^{[1],\vee}\;
\text{is}\;
\begin{dcases*}
\text{an isomorphism} & if \(j = i + 1\), and \\
\text{zero} & otherwise.
\end{dcases*}
\]
Indeed, if \(V\) has a basis with Gram matrix \(\mathbf{N}_m^{\oplus b}\),
set \(V_i\) to be the span of the vectors corresponding to the \(i\)-th column
of each \(\mathbf{N}_m\) block. Conversely, given the vector space decomposition,
begin with a basis of \(V_1\) and use the maps \(\beta_{i,i+1}\) to construct a
basis of \(V\) with the desired Gram matrix.

The following is a more flexible characterization of such forms:

\begin{Lemma}\label{classification-basis}
Let \(V\) be an \(mb\)-dimensional vector space with decomposition
\(V = \bigoplus_{i = 1}^m V_i\) such that
\begin{enumerate}
\item\label{classification-basis.kernel}
\(V_1 = \mathrm{P}_1 V\) and
\(V_m^{[1]} = \mathrm{P}_1' V^{[1]}\),
\item\label{classification-basis.perfect}
\(\beta_{1,2} \colon V_2 \subseteq V \to V^{[1],\vee} \twoheadrightarrow V_1^{[1],\vee}\)
is an isomorphism, and
\item\label{classification-basis.images}
\(
  \image(\beta \colon V_{i+1} \to V^{[1],\vee}) =
  \image(\beta^{[1],\vee} \colon V_{i-1}^{[2]} \to V^{[1],\vee})\)
for each \(1 < i < m\).
\end{enumerate}
Then \(\beta\) has a Gram matrix given by \(\mathbf{N}_m^{\oplus b}\) with
\(b \coloneqq \dim_\kk V_1\).
\end{Lemma}

\begin{proof}
The task is to adjust the given decomposition of \(V\) so that the maps
\(\beta_{i,j}\) from \parref{classification-standard-form} are isomorphisms
when \(j = i + 1\) and zero otherwise. Post-composing the maps appearing in
\ref{classification-basis.images}
with the restriction \(V^{[1],\vee} \to V_j^{[1],\vee}\) shows that:
for all \(1 < i < m\) and \(1 \leq j \leq m\),
\[
\rank(\beta_{j,i+1} \colon V_{i+1} \to V_j^{[1],\vee}) =
\rank(\beta_{i-1,j} \colon V_j \to V_{i-1}^{[1],\vee}).
\]
The maps \(\beta \colon V_{i+1} \to V^{[1],\vee}\) and
\(\beta^\vee \colon V_{i-1}^{[1]} \to V^\vee\) are injections for each
\(1 < i < m\) by assumption \ref{classification-basis.kernel}, so the rank
equations combined with \ref{classification-basis.perfect} implies that each
subspace \(V_i\) is \(b\)-dimensional and that the maps \(\beta_{i,i+1}\) are
isomorphisms for each \(1 \leq i < m\). Since \(\beta_{j,1}\) and
\(\beta_{m,j}\) are the zero maps for all \(1 \leq j \leq m\) by
\ref{classification-basis.kernel}, the rank equations furthermore imply that
\[
\beta_{i,j} = 0
\;\text{whenever \(j \neq i+1\) and}\;
\begin{dcases*}
\phantom{\max\{i,j\}}\llap{\(\min\{i,j\}\)} \equiv \rlap{\(1\)}\phantom{m} \pmod{2}, \\
\max\{i,j\} \equiv m \pmod{2}.
\end{dcases*}
\]
Adjust now the even-indexed subspaces via a Gram--Schmidt-like procedure,
in a way depending on the parity of \(m\), to arrange for the remaining
\(\beta_{i,j}\) to vanish.

\smallskip
\noindent\textbf{Case \(m\) is odd.}
It remains to arrange for \(\beta_{i,j} = 0\)
when both \(i\) and \(j\) are even. Let \(1 \leq k < m/2\) and inductively
assume that \(\beta_{i,j} = 0\) whenever \(\min\{i,j\} < 2k\). The task is to
modify the subspaces \(V_i\) with \(i\) even and \(2k \leq i \leq m\) so that
vanishing furthermore holds when \(\min\{i,j\} = 2k\).

Consider the subspace \(V' \coloneqq \bigoplus_{i = 2k}^m V_i\) and let
\(\beta'\) be the restriction of \(\beta\) thereon. For
\(2k \leq i \leq m\), set \(V_i' \coloneqq V_i\) for odd \(i\), and
\[
V_i' \coloneqq
\begin{dcases*}
\mathrm{P}_1 V'
& if \(i = 2k\), and \\
\ker(\beta \colon V_{2k+1} \oplus V_i \to V_{2k}'^{[1],\vee})
& otherwise.
\end{dcases*}
\]
Since \(\beta_{i,i+1}^\vee\) is an isomorphism and \(\beta^\vee\) vanishes on
\(V_m\) by \ref{classification-basis.kernel}, it follows that
\(\mathrm{P}_1' V'^{[1]} = V_m^{[1]}\), so
\[
\dim_\kk V_{2k}'
= \dim_\kk \mathrm{P}_1 V'
= \dim_\kk \mathrm{P}_1' V'^{[1]}
= \dim_\kk V_m^{[1]}
= b.
\]

Observe that the projections \(V_i' \to V_i\) are isomorphisms for all \(i\):
This is clear for \(i\) odd. For \(i = 2k\), this is because
both \(V_{2k}'\) and \(V_{2k}\) are \(b\)-dimensional and, since
\(\beta_{j,j+1} \colon V_{j+1} \to V_j^{[1],\vee}\) are isomorphisms,
\[
\ker(V_{2k}' \to V_{2k}) =
\mathrm{P}_1 V' \cap \Big(\bigoplus\nolimits_{i = 2k+1}^m V_i\Big) =
\ker\Big(\beta \colon
\bigoplus\nolimits_{i = 2k+1}^m V_i
\to V'^{[1],\vee}\Big)
=
\{0\}.
\]
Since \(\beta_{2k,2k+1}\) is an isomorphism and
\(\beta_{j,2k+1} = 0\) for \(2k + 1 \leq j \leq m\), this furthermore
implies that the map \(V_{2k+1} \to V_{2k}'^{[1],\vee}\) induced by \(\beta\)
is an isomorphism, whence the projection \(V_i' \to V_i\) is an isomorphism for
\(i\) even and different from \(2k\).

Therefore the subspaces \(V_i'\) yield a new direct sum decomposition
\[
V' = V_{2k}' \oplus V_{2k+1}' \oplus \cdots \oplus V_m'.
\]
It is then straightforward to check that upon replacing the \(V_i\) with the
\(V_i'\) for \(2k \leq i \leq m\), the maps \(\beta_{i,j}\), in addition to
their previous properties, vanish whenever \(i\) and \(j\) are both even and
\(\min\{i,j\} = 2k\). This completes the proof in the case \(m\) is odd.

\smallskip
\noindent\textbf{Case \(m\) is even.}
It remains to arrange vanishing of \(\beta_{i,j}\) when the smaller index is
even, and the larger index odd. This can be done all at once: for each
\(1 \leq k \leq m/2\), set \(V_{2k-1}' \coloneqq V_{2k-1}\) and
\[
V_{2k}' \coloneqq
\ker\Big(\beta \colon
V_{2k} \oplus \Big(
\bigoplus\nolimits_{\ell = k+1}^{m/2} V_{2\ell}\Big) \to
\bigoplus\nolimits_{\ell = k+1}^{m/2} V_{2\ell-1}^{[1],\vee}
\Big).
\]
The projection \(V_{2k}' \to V_{2k}\) is an isomorphism since
\(\beta\) restricts to an isomorphism between the two sums appearing in its
definition. Thus this gives a new decomposition
\(V = \bigoplus\nolimits_{i = 1}^m V_i'\), and it is
straightforward to check that the maps
\[
\beta_{i,j}' \colon
V_j' \subseteq
V \xrightarrow{\beta}
V^{[1],\vee} \twoheadrightarrow V_i'^{[1],\vee}
\]
induced by \(\beta\) on this decomposition satisfy the vanishing
conditions of \parref{classification-standard-form}: First, the subspaces
\(\bigoplus_{k = 1}^{m/2} V_{2k-1}\) and \(\bigoplus_{k = 1}^{m/2} V_{2k}\) are totally
isotropic for \(\beta\), and so the maps \(\beta_{i,j}'\) vanish whenever \(i\) and
\(j\) are of the same parity. Next, it follows by construction of the \(V_i'\)
that \(\beta_{i,j}'\) vanishes whenever \(\min\{i,j\}\) is even
and \(\max\{i,j\}\) is odd. Thus the restriction of \(\beta\) to \(V_{2k}'\)
factors as
\[
V_{2k}'
\subset V_{2k} \oplus \Big(\bigoplus\nolimits_{\ell = k+1}^{m/2} V_{2\ell}\big)
\xrightarrow{\beta} \bigoplus\nolimits_{s = 1}^k V_{2s-1}^{[1],\vee}.
\]
Since \(\beta_{2s-1, 2\ell} = 0\) for each \(1 \leq s < \ell \leq m/2\), the image of
\(V_{2k}'\) coincides with that of \(V_{2k}\), and is the space \(V_{2k-1}^{[1],\vee}\).
This gives the remaining conditions on \(\beta_{i,j}'\), completing the proof.
\end{proof}

\subsectiondash{Seesawing between filtrations}\label{classification-seesaw}
Conditions \ref{classification-basis.perfect} and
\ref{classification-basis.images} of \parref{classification-basis}
will be arranged by studying the relationship between the \(\perp\)- and
\(\perp^{[\smallbullet]}\)-filtrations under \(\beta\) and its dual.
The fundamental relationship is the following
seesaw relationship between intersections of pieces of the two filtrations:

\begin{Lemma}\label{classification-propagate-subspace}
For integers \(i,j \geq 0\), there is an equality of subspaces of
\(V^{[j]}\) given by
\[
(\mathrm{P}_{i+1} V^{[j-1]} \cap \mathrm{P}_{j-1}' V^{[j-1]})^{\perp^{[j-1]}}
=
\mathrm{P}_i V^{[j]} + \mathrm{P}_j' V^{[j]}
=
(\mathrm{P}_{i-1} V^{[j+1]} \cap \mathrm{P}_{j+1}' V^{[j+1]})^{\perp^{[j]}}.
\]
In particular, there is an equality of subspaces of \(V^{\vee,[j]}\) given by
\[
\image\big(
  \beta^{[j-1]} \colon
  \mathrm{P}_{i+1} V^{[j-1]} \cap \mathrm{P}_{j-1}' V^{[j-1]} \to
  V^{\vee,[j]}\big)
=
\image\big(
  \beta^{[j],\vee} \colon
  \mathrm{P}_{i-1} V^{[j+1]} \cap \mathrm{P}_{j+1}' V^{[j+1]} \to
  V^{\vee,[j]}\big).
\]
\end{Lemma}

\begin{proof}
Since
\(V^{[j],\perp^{[j-1]}}
= \mathrm{P}_1 V^{[j-1]}
\subseteq \mathrm{P}_{i+1} V^{[j-1]}\)
and
\(
V^{[j],\perp^{[j]}}
= \mathrm{P}_1' V^{[j+1]}
\subseteq \mathrm{P}_{j+1}' V^{[j+1]}
\),
\parref{forms-orthogonal-sequence}\ref{forms-orthogonal-cap-is-sum} applies and shows that the two
orthogonals in the statement are given by the left and right sides of:
\[
\mathrm{P}_{i+1} V^{[j-1], \perp^{[j-1]}} + \mathrm{P}_{j-1}' V^{[j-1],\perp^{[j-1]}}
= \mathrm{P}_i V^{[j]} + \mathrm{P}_j' V^{[j]}
= \mathrm{P}_{i-1} V^{[j+1],\perp^{[j]}} + \mathrm{P}_{j+1}' V^{[j+1],\perp^{[j]}}.
\]
The two middle equalities follow from the commutation relationship for Frobenius
twists in \parref{forms-fr-twist-orthogonal} and reflexivity of orthogonals as in
\parref{forms-orthogonal-sequence}\ref{forms-orthogonal-reflexive}, yielding the first statement. The
second statement now follows, since the two images are dual to the quotient
of \(V^{[j]}\) by the common orthogonal.
\end{proof}

The following consequence of \parref{classification-propagate-subspace} with
\(i = 2\) is helpful in relating the recognition principle condition
\parref{classification-basis}\ref{classification-basis.perfect} with the
numerical invariants of \(\beta\) via \parref{forms-invariants-decreasing}:

\begin{Corollary}\label{classification-complementary-to-kernel}
Let \(m \geq 1\) be an integer and \(\epsilon \coloneqq (-1)^m\). Then
\(\beta^{[m-1],\vee}\) induces an exact sequence
\[
0 \to
\frac{\mathrm{P}_1 V^{[m]} \cap \mathrm{P}_{m-\epsilon-1}' V^{[m]}}
     {\mathrm{P}_1 V^{[m]} \cap \mathrm{P}_{m+\epsilon-1}' V^{[m]}} \to
\left(
  \frac{\mathrm{P}_{m+\epsilon-2}'V^{[m-1]}}
       {\mathrm{P}_{m-\epsilon-2}'V^{[m-1]}}
\right)^\vee \to
\left(
  \frac{\mathrm{P}_2 V^{[m-1]} \cap \mathrm{P}_{m+\epsilon-2}' V^{[m-1]}}
       {\mathrm{P}_2 V^{[m-1]} \cap \mathrm{P}_{m-\epsilon-2}' V^{[m-1]}}
\right)^\vee \to
0.
\]
\end{Corollary}

\begin{proof}
Applying \parref{classification-propagate-subspace} with \(i = 2\) shows that
\(\beta^{[j],\vee}\) induces a short exact sequence
\[
0 \to
\mathrm{P}_1 V^{[j+1]} \cap \mathrm{P}_1' V^{[j+1]} \to
\mathrm{P}_1 V^{[j+1]} \cap \mathrm{P}_{j+1}' V^{[j+1]} \to
\big(V^{[j]}/(\mathrm{P}_2 V^{[j]} + \mathrm{P}_j' V^{[j]})\big)^\vee \to
0.
\]
Comparing \(j = m-1\) and \(j = m-2\epsilon-1\) then gives the isomorphism
\[
\frac{\mathrm{P}_1 V^{[m]} \cap \mathrm{P}_{m-\epsilon-1}' V^{[m]}}
     {\mathrm{P}_1 V^{[m]} \cap \mathrm{P}_{m+\epsilon-1}' V^{[m]}}
\cong
\left(
  \frac{\mathrm{P}_2 V^{[m-1]} + \mathrm{P}_{m+\epsilon-2}' V^{[m-1]}}
       {\mathrm{P}_2 V^{[m-1]} + \mathrm{P}_{m-\epsilon-2}' V^{[m-1]}}
\right)^\vee
\]
which implies the claimed exact sequence.
\end{proof}

Finally, a semi-linear algebra statement used in the main step of the
Classification Theorem to produce an orthogonal complement. The following says
that given a \(q\)-linear map \(V \to W\) between vector spaces, under certain
conditions, any subspace of \(V\) admits a complement with linearly disjoint
image. The hypothesis below is necessary, as can be seen by considering any
surjective map.

\begin{Lemma}\label{classification-q-linear-complement}
Let \(V\) and \(W\) be vector spaces over \(\kk\), and let \(f \colon V^{[1]} \to W\)
be a linear map. Assume that \(\ker(f)\) descends to \(V\). Then any subspace
\(V'\) of \(V\) admits a complement \(V''\) such that
\[
f(V'^{[1]}) \cap f(V''^{[1]}) = \{0\}.
\]
\end{Lemma}

\begin{proof}
Let \(\tilde{f} \colon V \to W\) be the \(q\)-linear map obtained by
precomposing \(f\) with the canonical map \(V \to V^{[1]}\). The additive
subgroup \(K \coloneqq \ker(\tilde{f})\) is, in fact, a linear subspace of
\(V\), and is the Frobenius descent of \(\ker(f)\). Choose a complement
\(V_1''\) to \(V' + K\) in \(V\), and choose a complement \(V_2''\) to
\(V' \cap K\) in \(K\). Then \(V'' \coloneqq V_1'' \oplus V_2''\) is a
complement to \(V'\) in \(V\) that satisfies the condition: since
\[
f(V'^{[1]}) \cap f(V''^{[1]}) =
f\big((V'^{[1]} + \ker(f)) \cap V''^{[1]}\big) =
f\big(((V' + K) \cap V'')^{[1]}\big)
\]
and since the rightmost intersection is contained inside \(K^{[1]} = \ker(f)\),
the intersection is \(\{0\}\).
\end{proof}

The following explains how to split off orthogonal summands of standard type
\(\mathbf{N}_m\) once the first \(m\) steps of the
\(\perp^{[\smallbullet]}\)-filtration are visible on \(V\) over \(\kk\):

\begin{Proposition}\label{classification-peel}
Let \(m \geq 1\) be an integer and assume that the first \(m\) steps of
the \(\perp^{[\smallbullet]}\)-filtration descend to \(V\) over \(\kk\).
Then there exists an orthogonal decomposition
\[
\beta = \beta' \oplus \beta''
\]
where \(\beta'\) has a Gram matrix given by
\(\mathbf{N}_m^{\oplus b}\) for \(b \coloneqq b_m(V,\beta)\).
\end{Proposition}

\begin{proof}
When \(m = 1\), let
\(V' \coloneqq \mathrm{P}_1 V \cap \mathrm{P}_1' V\) and let
\(V''\) be any complementary subspace in \(V\). Restricting \(\beta\)
to these subspaces gives subforms \(\beta'\) and \(\beta''\), respectively,
satisfying the properties of the statement in this case. Thus, in the
remainder, assume that \(m \geq 2\).

Construct subspaces \(V_i' \subseteq V\) for \(1 \leq i \leq m\) as follows:
Let \(\epsilon \coloneqq (-1)^m\). By \parref{forms-invariants-decreasing}
and \parref{classification-complementary-to-kernel}, it is possible to choose a
\(b\)-dimensional subspace
\(V_1' \subseteq \mathrm{P}_1 V \cap \mathrm{P}'_{m-\epsilon-1} V\)
such that the map
\[
V_1'^{[1]} \xrightarrow{\sim}
\left(
\frac{\mathrm{P}_1 V \cap \mathrm{P}_{m-\epsilon-1}' V}
     {\mathrm{P}_1 V \cap \mathrm{P}_{m+\epsilon-1}' V}
\right)^{[1]}
\hookrightarrow
\left(
  \frac{\mathrm{P}_{m+\epsilon-2}'V}
       {\mathrm{P}_{m-\epsilon-2}'V}
\right)^\vee \hookrightarrow
\mathrm{P}_{m+\epsilon-2}'V^\vee
\]
induced by \(\beta^\vee\) is an isomorphism onto its image \(W_1\).
Since the kernel of
\(\beta^\vee \colon V^{[1]} \to \mathrm{P}_{m+\epsilon-2}' V^\vee\) is
\(\mathrm{P}_{m+\epsilon-1}' V^{[1]}\), which admits a descent to \(V\) by
hypothesis,
\parref{classification-q-linear-complement} applies to give a complement
\(V_1''\) to \(V_1'\) in \(V\) whose images under \(\beta^\vee\) are linearly
disjoint. Extend the image of \(V_1''^{[1]}\) under \(\beta^\vee\) to a
complement \(W_1'\) of \(W_1\) so that
\(\mathrm{P}_{m+\epsilon-2}' V^\vee = W_1 \oplus W_1'\).
View \(W_1\) as a quotient and let \(V_2' \subseteq \mathrm{P}_{m+\epsilon-2}' V\)
be the dual subspace. Then \(V_2'\) is disjoint from
\(\mathrm{P}_{m-\epsilon-2}' V\), the map
\(\beta^\vee \colon V_1'^{[1]} \to V_2'^\vee\) is an isomorphism, and
\[
V_2'^\perp
\coloneqq \ker(\beta^\vee \colon V^{[1]} \to V_2'^\vee)
= \ker(\beta^\vee \colon V^{[1]} \to W_1)
= V_1''^{[1]}.
\]
Starting from \(V_1'\) and \(V_2'\), apply the second part of
\parref{classification-propagate-subspace} to successively choose
\(b\)-dimensional subspaces
\[
\begin{dcases*}
V'_{2k-1} \subseteq
\mathrm{P}_{2k-1} V \cap \mathrm{P}_{m-\epsilon-2k+1}' V
& for \(1 \leq k \leq \lceil m/2 \rceil\), and \\
V'_{2k\phantom{-1}} \subseteq
\mathrm{P}_{2k-2} V \cap \mathrm{P}_{m+\epsilon-2k\phantom{-1}}' V
& for \(1 \leq k \leq \lfloor m/2 \rfloor\),
\end{dcases*}
\]
linearly disjoint from \(\mathrm{P}_{m+\epsilon-2k+1}'V\) and
\(\mathrm{P}_{m-\epsilon-2k}'V\), respectively, and such that there is
an equality of images
\(
\image(\beta \colon V_{i+1}' \to V^{[1],\vee})
=
\image(\beta^{[1],\vee} \colon V_{i-1}'^{[2]} \to V^{[1],\vee})
\)
for each \(1 < i < m\).

The \(V_i'\) are therefore disjoint and together span an \(mb\)-dimensional subspace
\[
V' \coloneqq \bigoplus\nolimits_{i = 1}^m V_i' \subseteq V.
\]
The restriction \(\beta'\) of \(\beta\) satisfies the hypotheses of
\parref{classification-basis} with respect to this decomposition:
\ref{classification-basis.perfect} follows from the choice of
\(V_1'\) and \(V_2'\), and \ref{classification-basis.images} follows
from the construction of the \(V_i'\) via
\parref{classification-propagate-subspace}. Since each \(V_i'\) is
\(b\)-dimensional, this implies that \(\beta'\) has corank \(b\). Since
\(V_1' \subseteq \mathrm{P}_1 V'\) and
\(V_m' \subseteq \mathrm{P}_1' V'\) by construction, comparing dimensions then
gives \ref{classification-basis.kernel}. Therefore
\(\beta'\) has a Gram matrix given by \(\mathbf{N}_m^{\oplus b}\).

It remains to observe that the total orthogonal of \(V'\) descends to a
complement \(V''\) in \(V\). Taking orthogonals and
applying
\parref{forms-orthogonal-inclusion-reversing}\ref{forms-orthogonal-inclusion-reversing.sum}
to the direct sum decomposition shows that the total orthogonal is
\[
V'^{\perp}
\cap
V'^{[2],\perp^{[1]}}
=
\Big(\bigcap\nolimits_{i = 1}^m V_i'^\perp \Big)
\cap
\Big(\bigcap\nolimits_{i = 1}^m V_i'^{[2],\perp^{[1]}}\Big).
\]
By construction, \(V_{i+1}'^{\perp} = V_{i-1}'^{[2],\perp^{[1]}}\)
for each \(1 < i < m\) and
\(V_1'^{\perp} = V^{[1]} = V_m'^{[2],\perp^{[1]}}\), so
\[
V'^{\perp}
\cap
V'^{[2],\perp^{[1]}}
=
V_2'^{\perp} \cap
\Big(
  \bigcap\nolimits_{i = 1}^{m-1} V_i'^{[2],\perp^{[1]}}
\Big)
=
\Big(
V_1'' \cap
\Big(\bigcap\nolimits_{i = 1}^{m-1} V_i'^{[1],\perp}\Big)
\Big)^{[1]}
\eqqcolon
V''^{[1]}
\]
where the middle equality is due to the choice of \(V_2'\). Thus the total
orthogonal admits a Frobenius descent to a subspace \(V''\) of \(V\) which
has codimension at most \(mb\). Since \(\beta'\) does not have a radical,
\(V'\) and \(V''\) are linearly disjoint, and it follows that they are
complementary subspaces of \(V\). Taking \(\beta''\) to be the restriction of
\(\beta\) to \(V''\) then gives the result.
\end{proof}

\begin{proof}[Proof of \parref{forms-classification-theorem}]
By definition of \(\nu\) from \parref{filtrations-nu}, the
\(\perp^{[\smallbullet]}\)-filtration descends to \(V^{[\nu]}\) over \(\kk\). Since
formation of the filtrations is compatible with orthogonal decompositions by
\parref{forms-filtrations-sums}, \parref{classification-peel} may be
successively applied for each integer \(m \geq 1\) to produce orthogonal
summands \(\beta_m\) with Gram matrices \(\mathbf{N}_m^{\oplus b_m}\).
Since only finitely many of the invariants \(b_m\) are nonzero, this eventually
terminates, and the remaining piece \(\beta_0\) is a nonsingular subform of
\(\beta\) lifting that on \(\mathrm{P}_+ V/\mathrm{P}_- V\) as in
\parref{forms-nondegenerate-limit}.
\end{proof}

\section{Automorphisms}\label{section-automorphisms}
This Section is concerned with automorphism group schemes of \(q\)-bic forms.
For nondegenerate forms, this recovers the classical unitary group, see
\parref{auts-examples-unitary}. In general, these schemes are
positive-dimensional and are often nonreduced, as seen in the examples
\parref{auts-examples-standard} and \parref{auts-examples-1a+N2b}. The main
observation in this Section is that the nonreduced structure reflects the
fact that the \(\perp^{[\smallbullet]}\)-filtration canonically exists upon passing
to a Frobenius twist and that automorphisms need not preserve a descent of
the filtration; compare \parref{auts-preserve-filtrations} and
\parref{aut-reduced-preserves-perp'} with the main result
\parref{auts-smooth-parabolic} and its consequence
\parref{auts-identify-reduced}. This is used to compute the dimension of the
automorphism groups in \parref{aut-dimension}.

\subsectiondash{Automorphism group schemes}\label{forms-aut-schemes}
Let \((M,\beta)\) be a \(q\)-bic form over a \(\mathbf{F}_{q^2}\)-algebra \(R\).
Given an \(R\)-algebra \(S\), write \(M_S \coloneqq M \otimes_R S\) and
\(\beta_S \coloneqq \beta \otimes \id_S\) for the \(q\)-bic form over \(S\)
obtained by extension of scalars. Consider the group-valued functor
\[
\AutSch_{(M,\beta)} \colon
\mathrm{Alg}_R \to \mathrm{Grps}
\qquad
S \mapsto \Aut(M_S, \beta_S)
\]
that sends \(R\)-algebra \(S\) to the automorphism group of
\((M_S,\beta_S)\). This is the subfunctor of linear automorphisms
\(\mathbf{GL}_M\) of \(M\) which stabilizes the element
\(\beta \in \Hom_R(M^{[1]} \otimes_R M, R)\). Since \(M\) is always
assumed to be finite projective, this is represented by a closed subgroup
scheme of \(\mathbf{GL}_M\) by \cite[II.1.2.4 and II.1.2.6]{DG}, and is
referred to as the \emph{automorphism group scheme of \((M,\beta)\)}.

Since the \(\perp\)- and \(\perp^{[\smallbullet]}\)-filtrations from
\parref{forms-canonical-filtration} and
\parref{forms-canonical-filtration-second} are
intrinsic to a \(q\)-bic form \((M,\beta)\), they are preserved by points of
the automorphism group scheme whenever the constructions are compatible with
scalar extension, for instance when each piece is a local direct summand.
In other words:

\begin{Lemma} \label{auts-preserve-filtrations}
Assume that the submodules \(\mathrm{P}_i M \subseteq M\) and
\(\mathrm{P}_i' M^{[i]} \subseteq M^{[i]}\) are local direct summands for
every integer \(i \geq 0\). Then \(\AutSch_{(M,\beta)}\) factors through
the closed subgroup scheme of \(\GL_M\) given by
\[
\pushQED{\qed}
\AutSch_{(M,\mathrm{P}_{\smallbullet} M, \mathrm{P}_{\smallbullet}' M^{[\smallbullet]})} \colon
S \mapsto
\Set{g \in \mathrm{GL}(M_S) |
g \cdot \mathrm{P}_i M_S = \mathrm{P}_i M_S
\;\text{and}\;
g^{[i]} \cdot \mathrm{P}_i M_S^{[i]} = \mathrm{P}_i M_S^{[i]}\;
\text{for each}\; i\geq 0}.
\qedhere
\popQED
\]
\end{Lemma}

\subsectiondash{}\label{auts-preserve-descent}
However, \(\mathbf{GL}_M\) acts on \(M^{[i]}\) through its
\(R\)-linear Frobenius morphism, which, in coordinates, raises each matrix
coefficient to its \(q^i\)-th power. Importantly, this means that infinitesimal
members of \(\AutSch_{(M,\beta)}\) might not preserve a descent
\(\mathrm{P}_{\smallbullet}' M\) of the \(\perp^{[\smallbullet]}\)-filtration
to \(M\). So, in the case that the \(\perp^{[\smallbullet]}\)-filtration
descends to \(M\) over \(R\), let
\[
\AutSch_{(M,\beta,\mathrm{P}'_{\smallbullet}M)} \colon
S \mapsto
\Set{g \in \Aut(M_S,\beta_S)
  | g \cdot \mathrm{P}_i' M_S = \mathrm{P}_i' M_S \;\text{for each}\; i \in \mathbf{Z}_{\geq 0}}
\]
be the closed subgroup scheme of \(\AutSch_{(M,\beta)}\) that additionally
preserves \(\mathrm{P}'_{\smallbullet} M\). Then, at least:

\begin{Lemma}\label{aut-reduced-preserves-perp'}
Assume that the \(\perp^{[\smallbullet]}\)-filtration descends to \(M\) and
that each piece \(\mathrm{P}_i'M\) is a local direct summand. Then, as
subschemes of \(\mathbf{GL}_M\), the reduced subscheme of
\(\AutSch_{(M,\beta)}\) is contained in
\(\AutSch_{(M,\beta,\mathrm{P}'_{\smallbullet} M)}\).
\end{Lemma}

\begin{proof}
It suffices to show that if \(R\) is reduced, then every
\(g \in \Aut(M,\beta)\) preserves \(\mathrm{P}'_{\smallbullet} M\). Since
the \(i\)-th Frobenius twist of \(g\) preserves \(\mathrm{P}'_i M^{[i]}\),
the image \(N \coloneqq g \cdot \mathrm{P}_i' M\) fits into a commutative
diagram
\[
\begin{tikzcd}
\mathrm{P}_i' M \rar["g"'] \dar &
N \rar \dar &
M/\mathrm{P}_i' M \dar \\
\mathrm{P}_i' M^{[i]} \rar["g^{[i]}"] &
\mathrm{P}_i' M^{[i]} \rar["0"] &
M^{[i]}/\mathrm{P}_i' M^{[i]}
\end{tikzcd}
\]
where the vertical maps are induced by the canonical map
\(M \to M^{[i]}\). Since
\(\mathrm{P}_i' M\) is a local direct summand of \(M\), all the modules are
finite projective; together with reducedness of \(R\), it follows that
all the vertical maps are injective. Commutativity of the right square implies
that the map \(N \to M/\mathrm{P}_i' M\) vanishes, and so
\(N \subseteq \mathrm{P}_i' M\), implying the result.
\end{proof}

The remainder of this Section is concerned with \(q\)-bic forms
\((V,\beta)\) over a field \(\kk\) and the dimension of their automorphism
group schemes. First, their Lie algebras are as follows:

\begin{Proposition}\label{forms-aut-tangent-space}
There is a canonical identification of \(\kk\)-vector spaces given by
\[
\Lie\AutSch_ {(V,\beta)}
\cong \Hom_\kk(V,\mathrm{P}_1 V).
\]
When the \(\perp^{[\smallbullet]}\)-filtration descends to \(V\), there
is a canonical identification
\[
\Lie\AutSch_{(V,\beta,\mathrm{P}'_{\smallbullet} V)}
\cong \Set{\varphi \in \Hom_\kk(V,\mathrm{P}_1 V) |
\varphi(\mathrm{P}'_i V) \subseteq \mathrm{P}_1 V \cap \mathrm{P}'_i V
\;\text{for each}\; i \in \mathbf{Z}_{\geq 0}}.
\]
\end{Proposition}

\begin{proof}
The Lie algebra of \(\AutSch_{(V,\beta)}\) is the subset of
\(\Aut(V_D,\beta_D)\), with \(D \coloneqq \kk[\epsilon]/(\epsilon^2)\),
consisting of automorphisms that restrict to the identity upon setting
\(\epsilon = 0\). Such elements may be written as \(\id + \epsilon \varphi\)
for a unique \(\kk\)-linear map \(\varphi \colon V \to V\). Since
\(\epsilon^2 = 0\), that \(\id + \epsilon \varphi\) preserves \(\beta_D\) means
that, for every \(u \in V_D^{[1]}\) and \(v \in V_D\),
\[
\beta_D(u,v) =
\beta_D\big((\id + \epsilon \varphi)^{[1]}(u), (\id + \epsilon\varphi)(v)\big) =
\beta_D(u,v) + \epsilon \beta(\bar{u},\varphi(\bar{v})),
\]
where \(\bar{u} \in V^{[1]}\) and \(\bar{v} \in V\) are the images of \(u\) and
\(v\) upon setting \(\epsilon = 0\). Since \(\bar{u}\) is arbitrary, this implies
that \(\varphi(\bar{v}) \in V^{[1],\perp} = \mathrm{P}_1 V\), giving
the first identification. This implies the second statement since
the Lie algebra of
\(\AutSch_{(V,\beta,\mathrm{P}'_{\smallbullet} V)}\) is the subspace of
\(\Lie\AutSch_{(V,\beta)}\) consisting of those \(\varphi\) which, in addition,
preserve the filtration \(\mathrm{P}'_{\smallbullet} V\).
\end{proof}

Before continuing, consider a few examples of automorphism groups and schemes
of \(q\)-bic forms over a field. Some more examples and explicit computations
may be found in
\cite[\href{https://arxiv.org/pdf/2205.05273.pdf\#chapter.3}{Chapter 3}]{thesis}.

\subsectiondash{Unitary groups}\label{auts-examples-unitary}
When \(\beta\) is nonsingular, \parref{forms-aut-tangent-space} implies that
\(\AutSch_{(V,\beta)} \eqqcolon \mathrm{U}(V,\beta)\) is an \'etale group
scheme over \(\kk\) which might be called the \emph{unitary group} of \(\beta\).
By \parref{forms-hermitian-diagonal} and the comments thereafter,
its group of points over a separable closure of \(\kk\) is isomorphic to the
classical finite unitary group \(\mathrm{U}_n(q)\), as in \cite[\S2.1]{ATLAS}.
Following Steinberg as \cite[Lecture 11]{Steinberg:Lectures},
the unitary group can also be described as \(\mathrm{U}(V,\beta) = \GL_V^{F_\beta}\),
the fixed points for the morphism of algebraic groups given by
\[
F_\beta \colon \GL_V \to \GL_V
\qquad
g \mapsto \beta^{-1} \circ g^{[1],\vee,-1} \circ \beta.
\]

\subsectiondash{Type \texorpdfstring{\(\mathbf{N}_n\)}{N_n}}
\label{auts-examples-standard}
The automorphism group of a form \(\beta\) with a Gram matrix \(\mathbf{N}_n\)
can be roughly determined by analyzing the proof of
\parref{classification-peel}: First, there is the choice of a basis vector for
\(V_1 = \mathrm{P}_1 V \cap \mathrm{P}_{n-\epsilon-1}' V\). This now uniquely
determines, via \parref{classification-propagate-subspace},
spaces \(V_i\) together with a basis vector whenever \(i\) has the same
parity as \(n\). Finally, when the parity of \(i\) and \(n\) are different,
the \(V_i\) and its basis are only determined up to an additive factor.
In particular, this gives \(\dim\AutSch_{(V,\beta)} = \lceil n/2 \rceil\).

In contrast, \parref{forms-aut-tangent-space} shows that
\(\dim_\kk\Lie\AutSch_{(V,\beta)} = n\), so these schemes are not reduced for any
\(n \geq 2\). The schematic structure can be explictly determined in low
dimensions; for example, \(\AutSch_{(V,\beta)}\) may be described for
\(n \in \set{2,3,4}\) as closed subschemes of \(\mathbf{GL}_n\) consisting of
matrices
\[
\begin{pmatrix}
\lambda & \epsilon_1 \\
0 & \lambda^{-q}
\end{pmatrix},
\quad
\begin{pmatrix}
\lambda & t                     & \epsilon_1 \\
0       & \lambda^{-q}          & 0 \\
0       & -\lambda^{q(q-1)} t^q & \lambda^{q^2}
\end{pmatrix},
\quad
\begin{pmatrix}
\lambda & \epsilon_3                & t             & \epsilon_1 \\
0       & \lambda^{-q}              & 0             & 0 \\
0       & \epsilon_2                & \lambda^{q^2} & -\lambda^{-q^2(q-1)}\epsilon_2^q \\
0       & -\lambda^{-q^2(q+1)}t^q   & 0             & \lambda^{-q^3}
\end{pmatrix}
\]
where \(\lambda \in \mathbf{G}_m\), \(t \in \mathbf{G}_a\), and
\(\epsilon_i \in \boldsymbol{\alpha}_{q^i}\), and when \(n = 4\),
the entries satisfy
\[
\epsilon_2^q t^q -
\lambda^{q(q^2-q+1)} \epsilon_2 -
\lambda^{q^3} \epsilon_3^q = 0.
\]

\subsectiondash{Type \texorpdfstring{\(\mathbf{1}^{\oplus a} \oplus \mathbf{N}_2^{\oplus b}\)}{1^a+N2^b}}
\label{auts-examples-1a+N2b}
One other situation in which \(\AutSch_{(V,\beta)}\) admits a reasonably neat
description is when \(\beta\) admits a Gram matrix of the form
\(\mathbf{1}^{\oplus a} \oplus \mathbf{N}_2^{\oplus b}\). Then the automorphism
group is a \(b^2\)-dimensional closed subgroup scheme of \(\GL_{a + 2b}\), and
is described in
\cite[\href{https://arxiv.org/pdf/2205.05273.pdf\#subsection.1.3.7}{\textbf{1.3.7}}]{thesis};
for \(a = b = 1\), this is the \(1\)-dimensional subgroup \(\mathbf{GL}_3\)
given by matrices of the form
\[
\begin{pmatrix}
\zeta & 0 & -\zeta \epsilon_2^q/\lambda^q \\
\epsilon_2 & \lambda & \epsilon_1 \\
0 & 0 & \lambda^{-q}
\end{pmatrix}
\]
where \(\lambda \in \mathbf{G}_m\), \(\epsilon_1 \in \boldsymbol{\alpha}_q\),
\(\epsilon_2 \in \boldsymbol{\alpha}_{q^2}\), and \(\zeta \in \boldsymbol{\mu}_{q+1}\).

\subsectiondash{Filtered automorphisms}\label{auts-filtered-maps}
The examples suggest that the nonreducedness of \(\AutSch_{(V,\beta)}\) arises
from its failure to preserve a descent of its
\(\perp^{[\smallbullet]}\)-filtration to \(V\). In particular, the subgroups
in \parref{auts-examples-standard} and \parref{auts-examples-1a+N2b} that
preserve \(\mathrm{P}_{\smallbullet}' V\) are in fact reduced. This turns
out to be general and is established below by studying infinitesimal
deformations of the identity automorphism of
\((V,\beta,\mathrm{P}_{\smallbullet}' V)\). In preparation, establish
some notation: write
\begin{align*}
\bar{V} & \coloneqq V/\mathrm{P}_1 V,
&
\mathrm{P}_{\smallbullet} \bar{V} & \coloneqq
\mathrm{P}_{\smallbullet} V/\mathrm{P}_1 V,
&
\mathrm{P}_{\smallbullet}' \bar{V} & \coloneqq
\mathrm{P}_{\smallbullet}' V/\mathrm{P}_1 V \cap \mathrm{P}_{\smallbullet}' V, \\
\bar{V}' & \coloneqq V/\mathrm{P}_1' V,
&
\mathrm{P}_{\smallbullet}' \bar{V}' & \coloneqq
\mathrm{P}_{\smallbullet}' V/\mathrm{P}_{\smallbullet} V \cap \mathrm{P}_1' V,
&
\mathrm{P}_{\smallbullet}' \bar{V}' & \coloneqq
\mathrm{P}_{\smallbullet}' V/\mathrm{P}_1' V
\end{align*}
for the quotients of \(V\) by the two kernels of \(\beta\) and for the
quotient filtrations induced by \(\mathrm{P}_{\smallbullet} V\) and
\(\mathrm{P}_{\smallbullet}' V\). Scalar extensions of these spaces
are denoted, as usual, via subscripts. Write \(\AutSch_{(V,\mathrm{P}_{\smallbullet} V, \mathrm{P}'_{\smallbullet} V)}\)
for the algebraic group of linear automorphisms of \(V\) that preserve the
two filtrations \(\mathrm{P}_{\smallbullet} V\) and \(\mathrm{P}'_{\smallbullet} V\),
and analogously with \(V\) replaced by \(\bar{V}\) and \(\bar{V}'\).
Then the pairing induced by \(\beta\) on the quotients satisfies:

\begin{Lemma}\label{auts-frobenius-seesaw}
The perfect pairing
\(\bar{\beta} \colon \bar{V}'^{[1]} \otimes_\kk \bar{V} \to \kk\) induces an
isomorphism
\[
\bar{\beta}_* \colon
(\AutSch_{(\bar{V}', \mathrm{P}_{\smallbullet} \bar{V}', \mathrm{P}_{\smallbullet}' \bar{V}')})^{[1]} \to
\AutSch_{(\bar{V}, \mathrm{P}_{\smallbullet} \bar{V}, \mathrm{P}_{\smallbullet}' \bar{V})}
\]
of linear algebraic groups over \(\kk\) whose value on an \(S\)-point
\(\varphi \colon \bar{V}_S'^{[1]} \to \bar{V}_S'^{[1]}\) is the composite
\[
\bar{\beta}_S^{-1} \circ \varphi^{\vee,-1} \circ \bar{\beta}_S
\colon
\bar{V}_S \to
\bar{V}_S'^{[1],\vee} \to
\bar{V}_S'^{[1],\vee} \to
\bar{V}_S.
\]
\end{Lemma}

\begin{proof}
It remains to observe that \(\bar{\beta}_{*,S}(\varphi)\) defined as above
preserves the two filtrations on \(\bar{V}_S\). This follows from the seesaw
relation between the filtrations: \parref{classification-propagate-subspace}
implies that \(\bar{\beta}\) defines an isomorphism
\[
\bar{\beta} \colon
\mathrm{P}_{i+1} \bar{V} \cap \mathrm{P}_{j-1}' \bar{V} \xrightarrow{\sim}
\big(\bar{V}'/(\mathrm{P}_i \bar{V}' + \mathrm{P}_j' \bar{V}')\Big)^{[1],\vee}
\]
for all \(i, j \geq 0\), compatible with scalar extension. Then since
\(\varphi\) preserves the two filtrations on \(\bar{V}'^{[1]}_S\), its transpose
inverse induces an automorphism on the right-hand quotient. Transporting this
via \(\bar{\beta}^{-1}\) shows that \(\bar{\beta}_{*,S}(\varphi)\) preserves
the two filtrations on \(\bar{V}_S\).
\end{proof}

Let \(\mathrm{Art}_\kk\) be the category of Artinian \(\kk\)-algebras with
residue field \(\kk\). Let a \emph{\(q\)-small extension} denote a surjection
\(B \to A\) in \(\mathrm{Art}_\kk\) whose kernel is annihilated by the
\(q\)-power Frobenius. The following construction produces unique lifts, up to
Frobenius, along \(q\)-small extensions:

\begin{Lemma}\label{aut-small-lift-along-frob}
Let \(B \to A\) be a \(q\)-small extension. For any \(\kk\)-scheme \(X\) such
that \(X(B) \to X(A)\) is surjective, there exists a canonical map
\(\phi \colon X(A) \to X^{[1]}(B)\) fitting into a commutative diagram
\[
\begin{tikzcd}
X(B) \rar["\Fr_{X/\kk}"] \dar & X^{[1]}(B) \dar \\
X(A) \rar["\Fr_{X/\kk}"'] \ar[ur,"\phi"] & X^{[1]}(A)\punct{.}
\end{tikzcd}
\]
\end{Lemma}

\begin{proof}
Given \(x \in X(A)\), let \(\phi(x) \coloneqq y \circ \Fr_{X/\kk}\) for
any lift \(y \in X(B)\). This is well-defined because, if \(y'\) is another
lift, then the difference \(y^\# - y'^\#\) of the induced maps on structure
sheaves factors through \(\ker(B \to A)\) and is therefore annihilated by
Frobenius. The top triangle now commutes by construction, and the bottom
triangle because Frobenius commutes with ring homomorphisms.
\end{proof}

The following is the main result of this Section:

\begin{Theorem}\label{auts-smooth-parabolic}
Let \((V,\beta)\) be a \(q\)-bic form over a field \(\kk\) that admits a
descent \(\mathrm{P}'_{\smallbullet} V\) of its
\(\perp^{[\smallbullet]}\)-filtration to \(V\) over \(\kk\). Then the algebraic
group
\(\AutSch_{(V,\beta,\mathrm{P}'_{\smallbullet} V)}\) is reduced and smooth over
\(\kk\).
\end{Theorem}

\begin{proof}
It suffices to verify the infinitesimal lifting criterion, as in \citeSP{02HX},
at the identity. Explicitly, consider the functor of infinitesimal deformations
of the identity automorphism:
\[
G \colon
\mathrm{Art}_\kk \to \mathrm{Grps}
\qquad
A \mapsto
\Set{g \in \Aut(V_A,\beta_A,\mathrm{P}'_{\smallbullet} V_A) |
g \otimes_A A/\mathfrak{m}_A = \id_V}.
\]
Then it is enough to show that \(G(B) \to G(A)\) is surjective for every
\(q\)-small extension \(B \to A\).
Fix such an extension and fix an element \(g \in G(A)\). A lift
\(h \in G(B)\) will be constructed in several steps. The notation of
\parref{auts-filtered-maps} will be used throughout the proof.

\textbf{Step 1.} Forget \(\beta_A\) and view \(g\) as an element of the group
\(\Aut(V_A,\mathrm{P}_{\smallbullet} V_A, \mathrm{P}'_{\smallbullet} V_A)\).
In particular, \(g\) preserves the submodules
\(\mathrm{P}_1 V_A\) and \(\mathrm{P}_1' V_A\), so it induces automorphisms on
the quotients
\[
\bar{g} \in
\Aut(\bar{V}_A, \mathrm{P}_{\smallbullet} \bar{V}_A, \mathrm{P}'_{\smallbullet} \bar{V}_A)
\quad\text{and}\quad
\bar{g}'
\in \Aut(\bar{V}_A', \mathrm{P}_{\smallbullet} \bar{V}_A', \mathrm{P}'_{\smallbullet} \bar{V}_A').
\]
Apply \parref{aut-small-lift-along-frob} to the smooth algebraic group
\(X = \AutSch_{(\bar{V}',\mathrm{P}_{\smallbullet} \bar{V}', \mathrm{P}'_{\smallbullet} \bar{V}')}\)
of bifiltered linear automorphisms of \(\bar{V}'\)
to obtain a homomorphism
\[
\phi \colon
\Aut(\bar{V}'_A, \mathrm{P}_{\smallbullet} \bar{V}'_A, \mathrm{P}'_{\smallbullet} \bar{V}'_A) \to
\Aut(\bar{V}_B'^{[1]}, \mathrm{P}_{\smallbullet} \bar{V}_B'^{[1]}, \mathrm{P}'_{\smallbullet} \bar{V}_B'^{[1]})
\]
factoring the relative Frobenius homomorphism of \(X\) on \(B\)-points. Identify
the target of \(\phi\) with the group of bifiltered automorphisms of \(\bar{V}_B\)
via \parref{auts-frobenius-seesaw}. This yields an element
\[
\bar{h} \coloneqq \bar{\beta}_{B,*}(\phi(\bar{g}'))
\in
\Aut(\bar{V}_B,\mathrm{P}_{\smallbullet} \bar{V}_B, \mathrm{P}'_{\smallbullet} \bar{V}_B).
\]
By commutativity of the diagram in \parref{aut-small-lift-along-frob},
the definition of the isomorphism \(\bar{\beta}_*\) from \parref{auts-frobenius-seesaw},
and the fact that \(g\) preserves \(\beta_A\), it follows that \(\bar{h}\)
reduces along \(B \to A\) to
\[
\bar{\beta}_{A,*}(\bar{g}'^{[1]}) =
\bar{g} \in
\Aut(\bar{V}_A, \mathrm{P}_{\smallbullet} \bar{V}_A, \mathrm{P}' \bar{V}_A).
\]

\textbf{Step 2.} The quotient map \(V \to \bar{V}\) induces a surjection of
linear algebraic groups
\[
\AutSch_{(V,\mathrm{P}_{\smallbullet} V, \mathrm{P}_{\smallbullet}' V)} \to
\AutSch_{(\bar{V}, \mathrm{P}_{\smallbullet} \bar{V}, \mathrm{P}_{\smallbullet}' \bar{V})}
\]
whose kernel is isomorphic to the group of linear maps \(V \to \mathrm{P}_1 V\)
which preserve the filtrations induced by \(\mathrm{P}_{\smallbullet}' V\) and
restrict to an isomorphism on \(\mathrm{P}_1 V \subseteq V\). The surjection
is therefore smooth and, in terms of deformation theory, this means
that the canonically induced map of sets
\[
\Aut(V_B, \mathrm{P}_{\smallbullet} V_B, \mathrm{P}_{\smallbullet}' V_B) \to
\Aut(V_A, \mathrm{P}_{\smallbullet} V_A, \mathrm{P}_{\smallbullet}' V_A)
\times_{\Aut(\bar{V}_A, \mathrm{P}_{\smallbullet} \bar{V}_A, \mathrm{P}'_{\smallbullet} \bar{V}_A)}
\Aut(\bar{V}_B, \mathrm{P}_{\smallbullet} \bar{V}_B, \mathrm{P}'_{\smallbullet} \bar{V}_B)
\]
is surjective. See, for example, \citeSP{06HJ}. Step 1 produced an element
\((g \mapsto \bar{g} \mapsfrom \bar{h})\) of the fibre product on the right;
let \(h \in \Aut(V_B,\mathrm{P}_{\smallbullet} V_B, \mathrm{P}_{\smallbullet}' V_B)\)
be any lift along this surjection.

\textbf{Step 3.} It remains to see that this element \(h\) preserves the
\(q\)-bic form \(\beta_B\). Since \(h\) is a lift of \(g\), and since reduction
along \(B \to A\) commute with the quotient morphisms induced by
\(V \to \bar{V}'\), the image \(\bar{h}'\) of \(h\) in
\(\Aut(\bar{V}_B', \mathrm{P}_{\smallbullet} \bar{V}_B', \mathrm{P}_{\smallbullet}' \bar{V}_B')\)
is a lift of the element \(\bar{g}'\) from Step 1. Therefore the diagram of
\parref{aut-small-lift-along-frob} gives
\[
\bar{h}'^{[1]} = \phi(\bar{g}')
\in \Aut(\bar{V}_B'^{[1]}, \mathrm{P}_{\smallbullet} \bar{V}_B'^{[1]}, \mathrm{P}_{\smallbullet}' \bar{V}_B'^{[1]}).
\]
By the definitions of the perfect pairing \(\bar{\beta}_B\), of \(h\), and
of the isomorphism \(\bar{\beta}_{B,*}\) from \parref{auts-frobenius-seesaw},
it follows that, for any \(u \in V_B^{[1]}\) and \(v \in V_B\),
\[
\beta_B(h^{[1]} \cdot u, h \cdot v)
= \bar{\beta}_B(\bar{h}'^{[1]} \cdot \bar{u}, \bar{h} \cdot \bar{v})
= \bar{\beta}_B(\phi(\bar{g}')^{-1} \cdot \phi(\bar{g}') \cdot \bar{u},\bar{v})
= \bar{\beta}_B(\bar{u}, \bar{v})
= \beta_B(u,v),
\]
where \(\bar{u} \in \bar{V}_B'^{[1]}\) and \(\bar{v} \in \bar{V}_B\) are the
images of \(u\) and \(v\), respectively. Therefore \(h\) preserves \(\beta_B\).
Together with its properties from Step 2, this shows that \(h\) lies in
\(G(B)\) and is a lift of \(g \in G(A)\).
\end{proof}

As a first consequence, this gives a modular interpretation to the
reduced subscheme of \(\AutSch_{(V,\beta)}\):

\begin{Corollary}\label{auts-identify-reduced}
If the \(\perp^{[\smallbullet]}\)-filtration descends to \(V\) over \(\kk\), then
\(\AutSch_{(V,\beta),\mathrm{red}}
= \AutSch_{(V,\beta,\mathrm{P}_{\smallbullet}' V)}\).
\end{Corollary}

\begin{proof}
That ``\(\subseteq\)'' is \parref{aut-reduced-preserves-perp'}.
That ``\(\supseteq\)'' follows from \parref{auts-smooth-parabolic}: since
\(\AutSch_{(V,\beta,\mathrm{P}_{\smallbullet}' V)}\) is reduced, its closed
immersion into \(\AutSch_{(V,\beta)}\) factors through the reduced subscheme.
\end{proof}

Second, the following shows that the exponent of the nilpotent members in the
automorphism group scheme is determined by the number of Frobenius twists
required before the \(\perp^{[\smallbullet]}\)-filtration is canonically
defined; compare with \parref{filtrations-nu}:

\begin{Corollary}\label{aut-kill-infinitesimals}
Let \(\nu\) be such that every piece of the
\(\perp^{[\smallbullet]}\)-filtration of \((V,\beta)\) can be defined on
\(V^{[\nu]}\) in terms of \(\beta^{[\nu]}\). Then the quotient of
\(\AutSch_{(V,\beta)}\) by its \(q^\nu\)-power Frobenius kernel is reduced.
\end{Corollary}

\begin{proof}
Comparing \parref{auts-preserve-filtrations} and
\parref{auts-preserve-descent}, the hypothesis implies that the \(q^\nu\)-power
Frobenius of \(\AutSch_{(V,\beta)}\) factors through
\(\AutSch_{(V^{[\nu]}, \beta^{[\nu]}, \mathrm{P}'_{\smallbullet} V^{[\nu]})}\).
Since the latter is reduced by \parref{auts-smooth-parabolic}, so is the image
of the former.
\end{proof}

Finally, the main application is to compute dimensions of automorphism groups
of \(q\)-bic forms in terms of the numerical invariants defined in
\parref{forms-numerical-invariants-type}:

\begin{Theorem}\label{aut-dimension}
Let \((V,\beta)\) be a \(q\)-bic form over a field \(\kk\) of type
\((a;b_m)_{m \geq 1}\). Then
\[
\dim\AutSch_{(V,\beta)} =
\sum\nolimits_{k \geq 1}\Big[ k(b_{2k-1}^2 + b_{2k}^2) +
\Big(a + \sum\nolimits_{m \geq 2k} m b_m\Big) b_{2k-1} +
2k\Big(\sum\nolimits_{m \geq 2k+1} b_m\Big) b_{2k}\Big].
\]
\end{Theorem}

\begin{proof}
If \(\beta\) is nondegenerate, then \(\dim\AutSch_{(V,\beta)} = 0\) follows
from the Lie algebra computation \parref{forms-aut-tangent-space}. Otherwise,
replace \((V,\beta)\) by a sufficiently large Frobenius twist to assume that
the \(\perp^{[\smallbullet]}\)-filtration descends over \(\kk\); this is
possible by \parref{forms-frobenius-descent-index}, and allowable since the
automorphism group scheme of the twist of \((V,\beta)\) is the twist of the
automorphism group scheme \((V,\beta)\) and are therefore of the same
dimension. Since dimension is insensitive to nilpotents,
\parref{auts-identify-reduced} gives
\[
\dim\AutSch_{(V,\beta)} =
\dim\AutSch_{(V,\beta), \mathrm{red}} =
\dim\AutSch_{(V,\beta,\mathrm{P}'_{\smallbullet} V)}.
\]
Since \(\AutSch_{(V,\beta,\mathrm{P}'_{\smallbullet} V)}\) is smooth by
\parref{auts-smooth-parabolic}, its dimension is that of its Lie algebra.
With a choice of identification between \(V\) and its associated graded
for the filtration \(\mathrm{P}'_{\smallbullet} V\),
\parref{forms-aut-tangent-space} gives
\begin{multline*}
\Lie\AutSch_{(V,\beta,\mathrm{P}'_{\smallbullet} V)} \cong
\Big(
  \bigoplus\nolimits_{\ell \geq 1}
  \Hom_\kk(
    \mathrm{P}'_{2\ell-1} V/\mathrm{P}'_{2\ell-3} V,
    \mathrm{P}_1 V \cap \mathrm{P}'_{2\ell-1} V)
\Big) \\
\oplus
\Big(
  \Hom_\kk(
    \mathrm{P}'_+ V/\mathrm{P}'_- V,
    \mathrm{P}_1 V \cap \mathrm{P}'_+ V)
\Big)
\oplus
\Big(
  \bigoplus\nolimits_{\ell \geq 1}
  \Hom_\kk(
    \mathrm{P}'_{2\ell-2} V/\mathrm{P}'_{2\ell} V,
    \mathrm{P}_1 V \cap \mathrm{P}'_{2\ell-2} V)
\Big)
\end{multline*}
where \(\mathrm{P}_-' V\) and \(\mathrm{P}_+' V\) are, analogous to their
unprimed counterparts from \parref{forms-canonical-filtration}, the limits of
the increasing odd and decreasing even pieces of
\(\mathrm{P}_{\smallbullet}'V\), respectively. Express the dimension of each
parenthesized term in terms of \(a\) and the \(b_i\) using the formulae from
\parref{forms-numerical-invariants} and
\parref{forms-numerical-invariants-type}, and the symmetry relation
\parref{forms-intersect-filtrations}: The easiest is the central summand, with
dimension
\[
a \dim_\kk \mathrm{P}_1 V \cap \mathrm{P}_+' V =
a \sum\nolimits_{k \geq 1} b_{2k-1}.
\]
Next, the first parenthesized term has dimension
\begin{multline*}
\sum\nolimits_{\ell \geq 1} a_{2\ell-1} \dim_\kk \mathrm{P}_1 V \cap \mathrm{P}_{2\ell-1}' V
=
\sum\nolimits_{\ell \geq 1}
\Big(\sum\nolimits_{m \geq 2\ell-1} b_m\Big)
\Big(\sum\nolimits_{k = 1}^\ell b_{2k-1}\Big) \\
=
\sum\nolimits_{k \geq 1}
  \Big(\sum\nolimits_{\ell \geq k}\sum\nolimits_{m \geq 2\ell - 1} b_m\Big)
  b_{2k-1}
=
\sum\nolimits_{k \geq 1}
  \Big(\sum\nolimits_{\ell \geq k} (\ell - k +1) (b_{2\ell-1} + b_{2\ell})\Big) b_{2k-1}.
\end{multline*}
Finally, the third parenthesized term has dimension
\begin{multline*}
\sum\nolimits_{\ell \geq 1} a_{2\ell} \dim_\kk \mathrm{P}_1 V \cap \mathrm{P}_{2\ell-2}' V
=
\sum\nolimits_{\ell \geq 1}
  \Big(\sum\nolimits_{m \geq 2\ell} b_m\Big)
  \Big(\sum\nolimits_{k = 1}^\ell b_{2k - 1} + \sum\nolimits_{m \geq 2\ell} b_m\Big) \\
=
\sum\nolimits_{k \geq 1}
  \Big(\sum\nolimits_{\ell \geq k} (\ell - k + 1)(b_{2\ell} + b_{2\ell+1})\Big) b_{2k-1} +
\sum\nolimits_{\ell \geq 1}
  \Big(\sum\nolimits_{m \geq 2\ell} b_m\Big)^2.
\end{multline*}
The final sum of squares term here may be written as
\[
\sum\nolimits_{k \geq 1} k\Big(
\Big(b_{2k} + 2\sum\nolimits_{m \geq 2k+1} b_m\Big) b_{2k} +
\Big(b_{2k+1} + 2\sum\nolimits_{m \geq 2k+2} b_m\Big) b_{2k+1}\Big).
\]
Adding the expressions together gives the claimed formula.
\end{proof}

It is interesting and useful to disentangle this formula to see how dimensions
of automorphism groups grow under sums of \(q\)-bic forms:

\begin{Corollary}\label{moduli-dimension-summation}
Let \((V,\beta)\) and \((W,\gamma)\) be \(q\)-bic forms over \(\kk\). Then
\[
\dim\AutSch_{(V \oplus W, \beta \oplus \gamma)} =
\dim\AutSch_{(V,\beta)} +
\dim\AutSch_{(W,\gamma)} +
\Big(\sum\nolimits_{k \geq 1} b_{2k-1}\Big) c +
\sum\nolimits_{m \geq 1} \Phi_m(\beta) d_m
\]
where \(\beta\) and \(\gamma\) are of types
\((a; b_m)_{m \geq 1}\) and \((c; d_m)_{m \geq 1}\), respectively, and
\[
\Phi_m(\beta) \coloneqq
\begin{dcases*}
\dim_\kk V + b_{2k-1} + 2\sum\nolimits_{\ell = 1}^{k-1} (k-\ell) b_{2\ell-1}
& if \(m = 2k - 1\), and \\
\sum\nolimits_{\ell = 1}^{k-1} 2\ell b_{2\ell} + 2k\Big( \sum\nolimits_{\ell \geq 1} b_{2\ell-1} + \sum\nolimits_{\ell \geq k} b_{2\ell}\Big)
& if \(m = 2k\).
\end{dcases*}
\]
\end{Corollary}

\begin{proof}
This follows directly from \parref{aut-dimension} together with
the formulae of \parref{forms-numerical-invariants-type}.
\end{proof}

\section{Moduli}\label{section-moduli}
The parameter space for \(q\)-bic forms on a fixed \(n\)-dimensional vector
space \(V\) over the field \(\kk\) is given by the \(n^2\)-dimensional affine
space
\[
\qbics_V
\coloneqq \mathbf{A}(V^{[1]} \otimes_\kk V)^\vee
\coloneqq \Spec\Sym^*(V^{[1]} \otimes_\kk V).
\]
Multiplication in the symmetric algebra induces the universal \(q\)-bic form
\[
\beta_{\mathrm{univ}} \colon
V^{[1]} \otimes_\kk V \otimes_\kk \sO_{\qbics_V}
\to \sO_{\qbics_V}.
\]
In particular, \(\qbics_V\) represents the functor
\(\mathrm{Sch}_\kk^{\mathrm{opp}} \to \mathrm{Set}\) that sends
\[
X \mapsto \Set{\beta \colon V^{[1]} \otimes_\kk V \otimes_\kk \sO_X \to \sO_X
\;\text{a \(q\)-bic form over \(\sO_X\)}}
\]
a \(\kk\)-scheme \(X\) to the set of \(q\)-bic forms on \(V\) over \(\sO_X\). The
linear action of \(\GL_V\) on \(V^{[1]} \otimes_\kk V\) induces a schematic
action of the algebraic group \(\GL_V\) on \(\qbics_V\). By the Classification
Theorem \parref{forms-classification-theorem}, the orbits of this action consist
of the finitely many locally closed subschemes
\[
\qbics_{V,\mathbf{b}}
\coloneqq \Set{[\beta] \in \qbics_V | \operatorname{type}(\beta) = \mathbf{b}}
\]
parameterizing \(q\)-bic forms with a given type \(\mathbf{b} = (a; b_m)_{m \geq 1}\),
where \(a + \sum_{m \geq 1} m b_m = n\) as in \parref{forms-numerical-invariants-type}.
Together, these form the \emph{type stratification} of the affine space
\(\qbics_V\). This refines the natural stratification by corank.

Let \(\AutSch_{(V,\mathbf{b})}\) denote the automorphism group scheme of the
standard form of type \(\mathbf{b}\), as in \parref{forms-standard};
note that the automorphism group scheme of any other form of type
\(\mathbf{b}\) over \(\kk\) will be a form of \(\AutSch_{(V,\mathbf{b})}\) over \(\kk\).
Since the type strata are orbits under an algebraic group, they enjoy the
following standard properties, see, for example, \cite[Propositions 1.65, 1.66,
and 7.12]{Milne:AlgGroups}:

\begin{Proposition}\label{forms-aut-strata-dimension}
Each stratum \(\qbics_{V,\mathbf{b}}\) is a smooth, irreducible, locally closed
subscheme of \(\qbics_V\) of codimension \(\dim\AutSch_{(V,\mathbf{b})}\), and its
closure is a union of type stratum of strictly smaller dimension.
\qed
\end{Proposition}

This induces a partial ordering amongst types of \(q\)-bic forms on \(V\):
write \(\mathbf{b} \geq \mathbf{b}'\) if and only if the closure of
\(\qbics_{V,\mathbf{b}}\) contains \(\qbics_{V,\mathbf{b}'}\). Simple observations:
First, the maximal and minimal types are those corresponding to
\(\mathbf{1}^{\oplus n}\) and \(\mathbf{N}_1^{\oplus n}\), respectively.
Second, by comparing dimensions with \parref{aut-dimension}, it follows that
the maximal type in the codimension \(c^2\) locus of corank \(c\) forms is
\[
\begin{dcases*}
\mathbf{1}^{\oplus n-2c} \oplus \mathbf{N}_2^{\oplus c} &
if \(0 \leq c \leq n/2\), and \\
\mathbf{N}_1^{\oplus 2c - n} \oplus \mathbf{N}_2^{\oplus n - c} &
if \(n/2 \leq c \leq n\).
\end{dcases*}
\]
Third, by dividing out radicals, it follows that the subposet of types
containing \(\mathbf{N}_1\) is isomorphic to the poset of \(q\)-bic types on a
vector space of dimension \(1\) less.

The goal of
this Section is to characterize this partial ordering. As a first step,
rephrase this in terms of specialization relations for \(q\)-bic forms:
\(\mathbf{b} \geq \mathbf{b}'\) if and only if for some---equivalently, for
any---pair of \(q\)-bic forms \(\beta\) and \(\beta'\) on \(V\) of types
\(\mathbf{b}\) and \(\mathbf{b}'\), respectively, there exists a discrete
valuation ring \(R\) over \(\kk\), a \(q\)-bic form \((M,\gamma)\) over \(R\),
and isomorphisms
\[
(V_K, \beta_K) \cong (M_K, \gamma_K)
\quad\text{and}\quad
(V_\kappa, \beta_\kappa') \cong (M_\kappa, \gamma_\kappa)
\]
as \(q\)-bic forms over the fraction field \(K\) and residue field \(\kappa\)
of \(R\), respectively. Denote this situation by
\(\beta \rightsquigarrow \beta'\) and say that \emph{\(\beta\) specializes to
\(\beta'\)}.

The remainder of this Section will be phrased in terms of specialization
relations amongst \(q\)-bic forms on \(V\), and the goal is to determine
necessary and sufficient conditions for a specialization
\(\beta \rightsquigarrow \beta'\) to exist. A sequence of necessary
conditions is obtained by combining the summation formula
\parref{moduli-dimension-summation} with the fact that boundary strata in the
closure have smaller dimension:

\begin{Proposition}\label{moduli-specializations-necessary}
If there exists a specialization \(\beta \rightsquigarrow \beta'\), then
\(\Phi_m(\beta) \leq \Phi_m(\beta')\) for all \(m \geq 1\).
\end{Proposition}

\begin{proof}
Specializing subforms induces
\(\beta \oplus \gamma \rightsquigarrow \beta' \oplus \gamma\) for
all \(q\)-bic forms \((W,\gamma)\). Thus \parref{forms-aut-strata-dimension}
gives
\[
\dim \AutSch_{(V \oplus W, \beta  \oplus \gamma)} \leq
\dim \AutSch_{(V \oplus W, \beta' \oplus \gamma)}.
\]
According to \parref{moduli-dimension-summation}, the
dimensions of the automorphism groups grow linearly in the invariants of
\(\gamma\), so comparing coefficients of \(d_m\) gives the result.
\end{proof}

The following constructs a collection of basic specializations amongst
standard \(q\)-bic forms:

\begin{Lemma}\label{moduli-basic-specializations}
Let \(s \geq t \geq 1\) be integers. There exists specializations of \(q\)-bic
forms:
\begin{gather*}
\mathbf{N}_{2s+1}
  \rightsquigarrow \mathbf{1}^{\oplus 2} \oplus \mathbf{N}_{2s-1},
\qquad
\mathbf{N}_{2s}
  \rightsquigarrow \mathbf{1} \oplus \mathbf{N}_{2s - 1},
\qquad
\mathbf{1}^{\oplus 2} \oplus \mathbf{N}_{2s-2}
  \rightsquigarrow \mathbf{N}_{2s}, \\
\mathbf{N}_{2s-2t} \oplus \mathbf{N}_{2s+2}
  \rightsquigarrow \mathbf{N}_{2s-2t+2} \oplus \mathbf{N}_{2s},
\quad\quad
\mathbf{N}_{2s+1} \oplus \mathbf{N}_{2s+2t-1}
  \rightsquigarrow \mathbf{N}_{2s-1} \oplus \mathbf{N}_{2s+2t+1}.
\end{gather*}
\end{Lemma}

\begin{proof}
Let \(R\) be a discrete valuation ring over \(\kk\) with uniformizing
parameter \(\pi\), residue field \(\kk\), and fraction field \(K\). Let
\(M \coloneqq \bigoplus\nolimits_{i = 1}^n R \cdot e_i\) be a free \(R\)-module
of rank \(n\), and let \(\gamma\) be a \(q\)-bic form on \(M\) determined by
one of the following Gram matrices:
\[
\scalemath{0.75}{
\left[
  \begin{array}{@{}c|c@{}}
    \mathbf{N}_{2s-1} &
    {
    \begingroup
    \setlength\arraycolsep{0.2em}
      \begin{matrix}
        0 & 0 \\[-0.6em]
        \vdots & \vdots \\[-0.4em]
        \pi & 0
      \end{matrix}
    \endgroup}
    \\
    \hline
    {
    \begingroup
    \setlength\arraycolsep{0.2em}
      \begin{matrix}
      0 & \cdots & 0 \\[-0.4em]
      0 & \cdots & 0
      \end{matrix}
    \endgroup
    }
    &
    {
    \begingroup
    \setlength\arraycolsep{0.2em}
    \begin{matrix}
    0 & 1 \\[-0.4em]
    1 & 0
    \end{matrix}
    \endgroup}
  \end{array}
\right]},
\quad
\scalemath{0.75}{
\left[
  \begin{array}{@{}c|c@{}}
    \mathbf{N}_{2s-1} &
    \begin{matrix} 0 \\[-0.6em] \vdots \\[-0.4em] \pi \end{matrix} \\
    \hline
    {
    \begingroup
    \setlength\arraycolsep{0.2em}
    \begin{matrix}
    0 & \cdots & 0
    \end{matrix}
    \endgroup
    }
    &
    \begin{matrix}
     1
    \end{matrix}
  \end{array}
\right]},
\quad
\scalemath{0.75}{
\left[
  \begin{array}{@{}c|c@{}}
    \mathbf{N}_{2s-2} &
    {
    \begingroup
    \setlength\arraycolsep{0.2em}
      \begin{matrix}
        0 & 0 \\[-0.6em]
        \vdots & \vdots \\[-0.4em]
        1 & 0
      \end{matrix}
    \endgroup
    }\\
    \hline
    {
    \begingroup
    \setlength\arraycolsep{0.2em}
      \begin{matrix}
        0 & \cdots & 0 \\[-0.4em]
        0 & \cdots & 0
      \end{matrix}
    \endgroup
    }
    &
    {
    \begingroup
    \setlength\arraycolsep{0.2em}
      \begin{matrix}
        0 & 1 \\[-0.4em]
        \pi & 0
      \end{matrix}
    \endgroup
    }
  \end{array}
\right]},
\quad
\left[
  \begin{array}{@{}c|c|c@{}}
    \scalemath{0.8}{\mathbf{N}_{2s}} & &
    \scalemath{0.75}{
    \begingroup
    \setlength\arraycolsep{0.2em}
    \begin{matrix}
    0 & 0 \\[-0.6em]
    \vdots & \vdots \\[-0.4em]
    \pi & 0 \end{matrix}
    \endgroup} \\
    \hline
    & \scalemath{0.8}{\mathbf{N}_{2s-2t}}
    &
    \scalemath{0.75}{
    \begingroup
    \setlength\arraycolsep{0.2em}
    \begin{matrix}
      0 & 0 \\[-0.6em]
      \vdots & \vdots \\[-0.4em]
      1 & 0
    \end{matrix}
    \endgroup} \\
    \hline
    & &
    \scalemath{0.75}{
    \begingroup
    \setlength\arraycolsep{0.2em}
      \begin{matrix}
        0 & 1 \\[-0.4em]
        0 & 0
      \end{matrix}
    \endgroup}
  \end{array}
\right],
\quad
\left[
  \begin{array}{@{}c|c|c@{}}
    \scalemath{0.8}{\mathbf{N}_{2s-1}} & &
    \scalemath{0.75}{
    \begingroup
      \setlength\arraycolsep{0.2em}
      \begin{matrix}
        0      & 0      \\[-0.6em]
        \vdots & \vdots \\[-0.4em]
        \pi    & 0
        \end{matrix}
    \endgroup} \\
    \hline
    & \scalemath{0.8}{\mathbf{N}_{2s+2t-1}} &
    \scalemath{0.75}{
    \begingroup
      \setlength\arraycolsep{0.2em}
      \begin{matrix}
        0 & 0 \\[-0.6em]
        \vdots & \vdots \\[-0.4em]
        1 & 0
      \end{matrix}
    \endgroup} \\
    \hline
    & &
    \scalemath{0.75}{
    \begingroup
      \setlength\arraycolsep{0.2em}
      \begin{matrix}
        0 & 1 \\[-0.4em]
        0 & 0
      \end{matrix}
    \endgroup}
  \end{array}
\right].
\]
Then the reduction of \(\gamma\) modulo \(\pi\) yields a \(q\)-bic form of type
\[
\mathbf{1}^{\oplus 2} \oplus \mathbf{N}_{2s-1}, \quad
\mathbf{1} \oplus \mathbf{N}_{2s-1}, \quad
\mathbf{N}_{2s}, \quad
\mathbf{N}_{2s-2t+2} \oplus \mathbf{N}_{2s}, \quad
\mathbf{N}_{2s-1} \oplus \mathbf{N}_{2s+2t+1},
\]
respectively. To conclude, by \parref{forms-classification-theorem}, it
suffices to compute the \(\perp\)-filtration of
\(\beta \coloneqq \gamma \otimes_R K\) on
\(V \coloneqq M \otimes_R K\); do this by noting that, for each \(k \geq 1\),
\begin{equation}\label{moduli-basic-specializations.filtration}
\mathrm{P}_{2k-1} V
= \beta^{-1}\big((V/\mathrm{P}_{2k-2} V)^{[1],\vee}\big)
\quad\text{and}\quad
(V/\mathrm{P}_{2k} V)^\vee =
\image\big(\beta^\vee \colon \mathrm{P}_{2k-1} V^{[1]} \to V^\vee\big).
\end{equation}

Consider the first three Gram matrices. Comparing with
\parref{standard-forms-filtration-example} shows that
\(\mathrm{P}_1 V = \langle e_1 \rangle\), and that the maps
\(\beta \colon V \to V^{[1],\vee}\) and
\(\beta^\vee \colon V^{[1]} \to V^\vee\) satisfy
\[
\beta \colon e_i \mapsto e_{i-1}^{[1],\vee}
\quad\text{and}\quad
\beta^\vee \colon
e_i^{[1]} \mapsto e_{i+1}^\vee
\quad\text{for each}\;
1 \leq i \leq 2s-2,
\]
where \(e_0 \coloneqq 0\). Therefore
\(\mathrm{P}_{2k-1} V = \bigoplus\nolimits_{\ell = 1}^k K \cdot e_{2\ell-1}\)
and
\(V/\mathrm{P}_{2k} V \cong \bigoplus\nolimits_{\ell = 1}^k K \cdot e_{2\ell}\)
for each \(1 \leq k \leq s-1\).
Now consider each case in turn:

The first Gram matrix satisfies
\(\beta \colon e_{2s-1} \mapsto e_{2s-2}^{[1],\vee}\),
\(\beta^\vee \colon e_{2s-1}^{[1]} \mapsto \pi e_{2s}^\vee\), and
\(\beta \colon e_{2s+1} \mapsto e_{2s}^{[1],\vee}\).
Since \(\pi\) is invertible in \(K\), these combined with
\eqref{moduli-basic-specializations.filtration} respectively imply
\[
\mathrm{P}_{2s-1} V = \bigoplus\nolimits_{k = 1}^s K \cdot e_{2k-1},
\qquad
V/\mathrm{P}_{2s} V \cong \bigoplus\nolimits_{k = 1}^s K \cdot e_{2k},
\qquad
\mathrm{P}_{2s+1} V = \bigoplus\nolimits_{k = 1}^{s+1} K \cdot e_{2k-1},
\]
and so \(\beta\) is of type \(\mathbf{N}_{2s+1}\). A similar computation
applies to the second Gram matrix to show that, in that case,
\(\beta\) is of type \(\mathbf{N}_{2s}\).

The third Gram matrix satisfies
\(\beta \colon e_{2s-1} \mapsto e_{2s-2}^{[1],\vee} + \pi e_{2s}^{[1],\vee}\)
so neither \(e_{2s-2}^{[1],\vee}\) nor \(e_{2s}^{[1],\vee}\)
lies in the image of \(\beta\). Therefore the \(\perp\)-filtration has length
\(2s-2\) and
\[
\mathrm{P}_+ V/\mathrm{P}_- V =
\mathrm{P}_{2s-2} V/\mathrm{P}_{2s-3} V \cong K \cdot e_{2s-1} \oplus K \cdot e_{2s}
\]
meaning that \(\beta\) is of type \(\mathbf{1}^{\oplus 2} \oplus \mathbf{N}_{2s-2}\).

Completely analogous computations show that for the fourth and fifth
Gram matrices, \(\beta\) is of type
\(\mathbf{N}_{2s-2t} \oplus \mathbf{N}_{2s+2}\) and
\(\mathbf{N}_{2s+1} \oplus \mathbf{N}_{2s+2t-1}\), respectively.
\end{proof}

For each \(m \geq 1\) and \(q\)-bic form \(\beta\) of type
\((a;b_m)_{m \geq 1}\), write
\(\Theta_m(\beta) \coloneqq \sum_{k = 1}^m b_{2k-1}\).
The following gives a partial converse to
\parref{moduli-specializations-necessary}. The result is almost enough to
completely determine the closure relations in \(\qbics_V\) up to dimension \(6\): see
\cite[\href{https://arxiv.org/pdf/2205.05273.pdf\#subsection.3.1.2}{\textbf{3.1.2}},
\href{https://arxiv.org/pdf/2205.05273.pdf\#subsection.3.4.3}{\textbf{3.4.3}}, and
\href{https://arxiv.org/pdf/2205.05273.pdf\#subsection.3.8.2}{\textbf{3.8.2}}]{thesis}
for dimensions \(\leq 4\), and Figures \parref{moduli-five-dim-figure} and
\parref{moduli-six-dim-figure} for dimensions \(5\) and \(6\), respectively, where
\(\mathbf{0}\) is written in place of  \(\mathbf{N}_1\) for emphasis. See
\parref{moduli-specializations-remarks} for additional comments.

\begin{figure}
\[
\begin{tikzcd}[row sep=0.1em, column sep=0.8em]
  \mathbf{1}^6 \rar[symbol={\rightsquigarrow}]
& \mathbf{1}^4 \mathbf{N}_2 \rar[symbol={\rightsquigarrow}]
& \mathbf{1}^2 \mathbf{N}_4 \rar[symbol={\rightsquigarrow}] \ar[dr,symbol={\rightsquigarrow}]
& \mathbf{N}_6              \rar[symbol={\rightsquigarrow}] \ar[dr,symbol={\rightsquigarrow}]
& \mathbf{1} \mathbf{N}_5   \rar[symbol={\rightsquigarrow}]
& \mathbf{1}^3 \mathbf{N}_3 \rar[symbol={\rightsquigarrow}] \ar[dr,symbol={\rightsquigarrow}]
& \mathbf{0} \mathbf{1}^5   \ar[dr,symbol={\rightsquigarrow}]
\\
&
&
& \mathbf{1}^2 \mathbf{N}_2^2 \rar[symbol={\rightsquigarrow}]
& \mathbf{N}_2 \mathbf{N}_4   \ar[rr,symbol={\rightsquigarrow}]
&
& \mathbf{1}\mathbf{N}_2\mathbf{N}_3 \rar[symbol={\rightsquigarrow}] \ar[dr,symbol={\rightsquigarrow}] \ar[dd,symbol={\rightsquigarrow}]
& \mathbf{0}\mathbf{1}^3\mathbf{N}_2
\\
&
&
&
&
&
&
& \mathbf{N}_3^2 \rar[symbol={\rightsquigarrow}]
& \mathbf{0}\mathbf{N}_5
\\
&
&
&
&
&
& \mathbf{N}_2^3 \ar[rr,symbol={\rightsquigarrow}]
&
& \mathbf{0} \mathbf{1} \mathbf{N}_2^2
\end{tikzcd}
\]
\caption{Immediate specialization relations amongst for \(6\)-dimensional
\(q\)-bic forms, up to the first few with nontrivial radical.}
\label{moduli-six-dim-figure}
\end{figure}

\begin{Proposition}\label{moduli-specializations-sufficient}
Let \(\beta\) and \(\beta'\) be \(q\)-bic forms of type \((a;b_m)_{m \geq 1}\)
and \((a';b_m')_{m \geq 1}\), respectively. If
\(\Phi_m(\beta) \leq \Phi_m(\beta')\)
and \(\Theta_m(\beta) \leq \Theta_m(\beta')\) for all \(m \geq 1\), then
there exists a specialization \(\beta \rightsquigarrow \beta'\).
\end{Proposition}

\begin{proof}
For the arguments below, it is convenient to replace \(\Phi_m\) from
\parref{moduli-dimension-summation} by
\[
\Psi_m(\beta) \coloneqq
\begin{dcases*}
b_{2k-1} + 2\sum\nolimits_{\ell = 1}^{k-1} (k-\ell) b_{2\ell-1}
& if \(m = 2k - 1\), and \\
\sum\nolimits_{\ell = 1}^{k-1} \ell b_{2\ell} + k\Big(\sum\nolimits_{\ell \geq 1} b_{2\ell-1} + \sum\nolimits_{\ell \geq k} b_{2\ell}\Big)
& if \(m = 2k\),
\end{dcases*}
\]
in which redundant constants are removed, so that
\(\Phi_m(\beta) \leq \Phi_m(\beta')\) if and only if
\(\Psi_m(\beta) \leq \Psi_m(\beta')\).

The argument proceeds in steps. Each time,
\parref{moduli-basic-specializations} is applied to a subform of \(\beta\) to
produce an intermediate specialization \(\beta''\). The new form will satisfy
\(\Psi_m(\beta'') \leq \Psi_m(\beta')\) for all \(m\), together with
successively stronger conditions that, in particular, imply
\(\Theta_m(\beta'') \leq \Theta_m(\beta')\) for all \(m\). Replace
\(\beta\) by \(\beta''\) and continue in this manner until
\(\Psi_m(\beta) = \Psi_m(\beta')\) for all \(m\). This implies that
\(\beta\) and \(\beta'\) are of the same type, at which point the
Proposition follows from the Classification Theorem
\parref{forms-classification-theorem}.

In what follows, abbreviate
\(\Psi_m(\beta)\), \(\Psi_m(\beta')\), and \(\Psi_m(\beta'')\) to
\(\Psi_m\), \(\Psi_m'\), and \(\Psi''_m\);
similarly for \(\Theta_m\), etc.

\newcounter{specialization-steps} 
\smallskip
\textbf{Step \refstepcounter{specialization-steps}\arabic{specialization-steps}.}
\emph{After a sequence of specializations of the form
\[
\mathbf{N}_{2s+1} \rightsquigarrow
\mathbf{1}^{\oplus 2} \oplus
\mathbf{N}_{2s-1},
\]
may assume that \(b_{2k-1} \leq b_{2k-1}'\) for every \(k \geq 1\).}
\smallskip

The inequalities \(\Psi_{2k-1} \leq \Psi_{2k-1}'\) imply
that \(b_{2k-1}\) must be smaller than \(b_{2k-1}'\) the first time they
differ. This implies: if \(s\) is minimal such that
\(b_{2s+1} > b_{2s+1}'\), then \(\Psi_{2s-1} < \Psi_{2s-1}'\) and
\(\Theta_s < \Theta_s'\). The former inductively implies that
\(\Psi_{2k-1} < \Psi_{2k-1}'\) for every \(k \geq s\) since
\[
\Psi_{2k+1} - \Psi_{2k-1} =
\Theta_{k+1} + \Theta_{k-1} \leq
\Theta_{k+1}' + \Theta_{k-1}' =
\Psi_{2k+1}' - \Psi_{2k-1}'
\]
upon using the inequalities \(\Theta_m \leq \Theta_m'\) for \(m = k-1\) and
\(m = k+1\). The latter shows that this inequality is strict for \(k = s+1\),
implying that, furthermore, \(\Psi_{2k+1} + 1 < \Psi_{2k+1}'\) for all
\(k \geq s+1\). So if \(\beta''\) is obtained from \(\beta\) via a
specialization of the form
\(\mathbf{N}_{2s+1} \rightsquigarrow \mathbf{1}^{\oplus 2} \oplus \mathbf{N}_{2s-1}\),
then
\[
\Psi_m'' = \begin{dcases*}
\Psi_{2k-1} + 1 & if \(m = 2s-1\) or \(m = 2s+1\), \\
\Psi_{2k-1} + 2 & if \(m = 2k-1 > 2s+1\), \\
\Psi_m          & otherwise,
\end{dcases*}
\quad\text{and}\quad
\Theta_m'' = \begin{dcases*}
\Theta_s + 1 & if \(m = s\), \\
\Theta_m & otherwise,
\end{dcases*}
\]
and so with the inequalities above, this implies
\(\Psi_m'' \leq \Psi_m'\) and \(\Theta_m'' \leq \Theta_m'\) for all \(m\).

\smallskip
\textbf{Step \refstepcounter{specialization-steps}\arabic{specialization-steps}.}
\emph{After a sequence of specializations of the form
\[
\mathbf{N}_{2s} \rightsquigarrow
\mathbf{1}^{\oplus 2s-2t+1} \oplus \mathbf{N}_{2t-1}
\quad\text{or}\quad
\mathbf{1}^{\oplus 2t-2s-1} \oplus \mathbf{N}_{2s} \rightsquigarrow
\mathbf{N}_{2t-1},
\]
may assume that \(b_{2k-1} = b_{2k-1}'\) for every \(k \geq 1\).}

\smallskip
From now on, assume that \(b_{2k-1} \leq b_{2k-1}'\) for all \(k \geq 1\),
superseding the inequalities \(\Theta_m \leq \Theta_m'\). Note this also
implies \(\Psi_{2k-1} \leq \Psi_{2k-1}'\) for all \(k\), so it will suffice to
verify the inequalities \(\Psi_{2k} \leq \Psi_{2k}'\).

To begin this Step, suppose first that the inequalities
\begin{equation}\label{specialization-steps.even-domination}
\sum\nolimits_{k \geq m} b_{2k} \leq
\sum\nolimits_{k \geq m} b_{2k}'
\end{equation}
are satisfied for every \(m \geq 1\).
Comparing the formulae from \parref{forms-numerical-invariants-type} for
\(\dim_\kk V\) in terms of the invariants of
\(\beta\) and \(\beta'\), and successively applying the inequalities
\eqref{specialization-steps.even-domination} for increasing \(m\) gives
the inequality \begin{equation}\label{specialization-steps.compare-dimensions}
a
= a' + \sum\nolimits_{i \geq 1} i(b_i' - b_i)
\geq a' + \sum\nolimits_{k \geq 1} (2k-1)(b_{2k-1}' - b_{2k-1}).
\end{equation}
Let \(t\) be any index such that \(b_{2t-1} < b_{2t-1}'\). Then
\eqref{specialization-steps.compare-dimensions} implies
\(a \geq 2t-1\), meaning \(\beta\) contains a subform of type
\(\mathbf{1}^{\oplus 2t-1}\). Let \(\beta''\) be obtained via a specialization
\(\mathbf{1}^{\oplus 2t-1} \rightsquigarrow \mathbf{N}_{2t-1}\). The choice of
\(t\) ensures that \(b_{2k-1}'' \leq b_{2k-1}'\) for all \(k\). Combined with
the \(m = 1\) case of \eqref{specialization-steps.even-domination}, this
implies that
\[
\Psi_2'' =
\sum\nolimits_{i \geq 1} b_i'' \leq
\sum\nolimits_{i \geq 1} b_i' =
\Psi_2'.
\]
Then, for each \(k \geq 1\), the \(m = k+1\) case of
\eqref{specialization-steps.even-domination} gives inequalities
\begin{equation}\label{specialization-steps.even-differences}
\Psi_{2k+2}'' - \Psi_{2k}'' =
\sum\nolimits_{\ell \geq 1} b_{2\ell-1}'' +
\sum\nolimits_{\ell \geq k+1} b_{2\ell}''
\leq
\sum\nolimits_{\ell \geq 1} b_{2\ell-1}' +
\sum\nolimits_{\ell \geq k+1} b_{2\ell}' =
\Psi_{2k+2}' - \Psi_{2k}'.
\end{equation}
Starting from \(\Psi_2'' \leq \Psi_2'\), this inductively implies that
\(\Psi_{2k}'' \leq \Psi_{2k}'\) for all \(k \geq 1\).

It remains to consider the situation when at least one of the inequalities in
\eqref{specialization-steps.even-domination} fails. Since
\(b_{2k} = b_{2k}' = 0\) for large \(k\), there is a maximal \(s\) such that
\[
\sum\nolimits_{k \geq s} b_{2k} >
\sum\nolimits_{k \geq s} b_{2k}'.
\]
Comparing this with \eqref{specialization-steps.even-domination}
where \(m = s+1\) gives \(b_{2s} > b_{2s}'\). Now consider two cases:

\smallskip
\textbf{Case \arabic{specialization-steps}A.}
\emph{There exists \(t \leq s\) such that \(b_{2t-1} < b_{2t-1}'\).}

\smallskip
Let \(\beta''\) be any specialization of \(\beta\) obtained via
\(\mathbf{N}_{2s} \rightsquigarrow \mathbf{1}^{\oplus 2s-2t+1} \oplus \mathbf{N}_{2t-1}\).
Then \(\Psi_{2k}'' = \Psi_{2k} \leq \Psi_{2k}'\) for all \(k \leq s\).
Step 1 together with maximality of \(s\) implies that
\eqref{specialization-steps.even-differences} holds for all \(k \geq s\), which
then inductively implies that
\(\Psi_{2k}'' \leq \Psi_{2k}'\) also holds for \(k > s\).

\smallskip
\textbf{Case \arabic{specialization-steps}B.}
\emph{Every \(k\) for which \(b_{2k-1} < b_{2k-1}'\) satisfies \(k > s\).}

\smallskip
Applying the inequality \(\Psi_{2s}' - \Psi_{2s} \geq 0\) twice to the middle
term of \eqref{specialization-steps.compare-dimensions} gives the inequality
\[
a \geq
a'
+ \sum\nolimits_{k \geq s+1} (2k -2s)(b_{2k}' - b_{2k})
+ \sum\nolimits_{k \geq 1}(2k -2s-1)(b_{2k-1}' - b_{2k-1}).
\]
The first sum on the right is nonnegative by
\eqref{specialization-steps.even-domination} and maximality of \(s\). The
assumption in this case implies that the second sum is nonnegative, and that if
\(t\) is any index such that \(b_{2t-1} < b_{2t-1}'\), then in fact
\(a \geq 2t-2s-1 > 0\). In other words, \(\beta\) contains
\(\mathbf{1}^{\oplus 2t-2s-1} \oplus \mathbf{N}_{2s}\) as a subform;
let \(\beta''\) be obtained by specializing such a subform to \(\mathbf{N}_{2t-1}\).
Arguing as above shows that \(\Psi_{2k}'' \leq \Psi_{2k}'\) for all \(k\).

\smallskip
\textbf{Step \refstepcounter{specialization-steps}\arabic{specialization-steps}.}
\emph{After a sequence of specializations of the form
\[
\mathbf{1}^{\oplus 2} \oplus \mathbf{N}_{2s-2} \rightsquigarrow
\mathbf{N}_{2s}
\]
may assume that
\(\sum\nolimits_{k \geq 1} k b_{2k} = \sum\nolimits_{k \geq 1} k b_{2k}'\).
}

\smallskip
After Step 2, by passing to subforms, it may be assumed that
\(b_{2k-1} = b_{2k-1}' = 0\) for all \(k\).

Let \(s \geq 1\) be maximal such that \(\Psi_{2s-2} = \Psi_{2s-2}'\),
setting \(\Psi_0 = \Psi_0' = 0\) so that the maximum always exists. If
\(s = \max\set{k | b_{2k-2} \neq 0 \;\text{or}\; b_{2k-2}' \neq 0}\),
then this is equivalent to
\(\sum\nolimits_{k \geq 1} k b_{2k} = \sum\nolimits_{k \geq 1} k b_{2k}'\),
as required. Otherwise,
\(\sum\nolimits_{k \geq 1} k b_{2k} < \sum\nolimits_{k \geq 1} k b_{2k}'\),
so combined with the first equation in \eqref{specialization-steps.compare-dimensions},
\[
a = a' + 2\sum\nolimits_{k \geq 1} k (b_{2k}' - b_{2k}) \geq 2.
\]
Since \(\Psi_{2k} \leq \Psi_{2k}'\) for all \(k\) with strict inequality for
\(k \geq s\), it follows that
\[
b_{2s-2}
= 2\Psi_{2s-2} - \Psi_{2s-4} - \Psi_{2s}
> 2\Psi_{2s-2}' - \Psi_{2s-4}' - \Psi_{2s}'
= b_{2s-2}'.
\]
Thus, in this case,
\(\beta\) contains a subform of type \(\mathbf{1}^{\oplus 2} \oplus \mathbf{N}_{2s-2}\);
let \(\beta''\) be obtained by specializing it to a form of
type \(\mathbf{N}_{2s}\). Then \(\Psi_{2k}'' = \Psi_{2k}\) for
\(1 \leq k \leq s-1\) and \(\Psi_{2k}'' = \Psi_{2k} + 1\) for \(k \geq s\), so
maximality of \(s\), implies \(\Psi_{2k}'' \leq \Psi_{2k}'\) for all \(k\).

\smallskip
\noindent\textbf{Step \refstepcounter{specialization-steps}\arabic{specialization-steps}.}
\emph{After a sequence of specializations of the form
\[
\mathbf{N}_{2s} \rightsquigarrow
\mathbf{N}_2 \oplus \mathbf{N}_{2s-2}
\]
may assume that \(\sum\nolimits_{k \geq 1} b_{2k} = \sum\nolimits_{k \geq 1} b_{2k}'\).
}

\smallskip
Let \(s \geq 1\) be minimal such that \(\Psi_{2s} = \Psi_{2s}'\). If \(s = 1\), this
is equivalent to the desired equality. Otherwise, as in Step 3,
\[
b_{2s}
= 2\Psi_{2s} - \Psi_{2s-2} - \Psi_{2s+2}
> 2\Psi_{2s}' - \Psi_{2s-2}' - \Psi_{2s+2}'
= b_{2s}'.
\]
Therefore \(\beta\) contains a subform of type \(\mathbf{N}_{2s}\); specialize
it to \(\mathbf{N}_2 \oplus \mathbf{N}_{2s-2}\) to obtain \(\beta''\). Then
\(\Psi_{2k}'' = \Psi_{2k} + 1\) for \(1 \leq k < s\) and
\(\Psi_{2k}'' = \Psi_{2k}\) for \(k \geq s\), so the choice
of \(s\) implies \(\Psi_{2k}'' \leq \Psi_{2k}'\) for all \(k\).

\smallskip
\noindent\textbf{Step \refstepcounter{specialization-steps}\arabic{specialization-steps}.}
\emph{After a sequence of specializations of the form
\[
\mathbf{N}_{2t} \oplus \mathbf{N}_{2s} \rightsquigarrow
\mathbf{N}_{2t+2} \oplus \mathbf{N}_{2s-2}
\]
where \(s > t\), may assume that
\(\Psi_m = \Psi_m'\) for every \(m \geq 1\).
}

\smallskip
Consider any pair \(s > t\) such that
\(\Psi_{2t} = \Psi_{2t}'\),
\(\Psi_{2s} = \Psi_{2s}'\), and
\(\Psi_{2k} < \Psi_{2k}'\) for each
\(t < k < s\).
Arguing as in Steps 3 and 4, this implies that \(b_{2s} > b_{2s}'\) and
\(b_{2t} > b_{2t}'\), so that \(\beta\) contains
\(\mathbf{N}_{2t} \oplus \mathbf{N}_{2s}\); specialize this to
\(\mathbf{N}_{2t+2} \oplus \mathbf{N}_{2s-2}\) to obtain \(\beta''\). Then
\(\Psi_{2k}'' = \Psi_{2k} + 1\) for \(t < k < s\), and
\(\Psi_{2k}'' = \Psi_{2k}\) otherwise. By choice of \(s\)
and \(t\), this implies that
\(\Psi_{2k}'' \leq \Psi_{2k}'\) for all \(k\).
\end{proof}

\subsectiondash{Remarks}\label{moduli-specializations-remarks}
First, the hypotheses of \parref{moduli-specializations-sufficient}
may be relaxed: it is enough that there is some specialization
\(\beta \rightsquigarrow \beta''\) such that
\(\Theta_m(\beta'') \leq \Theta_m(\beta')\) for all \(m\). For instance, the
specialization
\(\mathbf{N}_3^{\oplus 2} \rightsquigarrow \mathbf{N}_1 \oplus \mathbf{N}_5\)
found at the right end of Figure \parref{moduli-six-dim-figure} does not
satisfy the inequality \(\Theta_2\),
but nonetheless exists by \parref{moduli-basic-specializations}; more generally,
any immediate specialization via
\[
\mathbf{N}_{2s+1} \oplus \mathbf{N}_{2s+2t-1} \rightsquigarrow
\mathbf{N}_{2s-1} \oplus \mathbf{N}_{2s+2t+1},
\]
the fifth basic specialization of \parref{moduli-basic-specializations}, will
violate some of the inequalities \(\Theta_m\). Note further that the proof of
\parref{moduli-specializations-sufficient} did not incorporate these
specializations at all.

Second, there are examples of \(q\)-bic forms \(\beta\) and \(\beta'\)
that satisfy the inequalities \(\Phi_m(\beta) \leq \Phi_m(\beta')\), but
cannot be specialized to one another via a sequence of basic specializations
from \parref{moduli-basic-specializations}. The first examples appear in
dimension \(15\) and corank \(3\); for instance, consider
\[
\beta \coloneqq \mathbf{1} \oplus \mathbf{N}_3^{\oplus 2} \oplus \mathbf{N}_8
\quad\text{and}\quad
\beta' \coloneqq \mathbf{N}_1 \oplus \mathbf{N}_7^{\oplus 2}.
\]
I do not know whether or not \(\beta\) specializes to \(\beta'\).

To go further, it would be interesting to study \(q\)-bic forms over
discrete valuation rings and the fine geometry of the closure of
\(\qbics_{V,\mathbf{b}}\). For the former, it may be helpful to reformulate
the canonical filtrations in terms of quotients so as to deal with jumps
in types upon specialization. For the latter, it would be interesting to
know the degrees and singularities of these varieties.

\bibliographystyle{amsalpha}
\bibliography{main}
\end{document}